\documentclass{amsart}
\usepackage{amsmath}
\usepackage{mathrsfs}
\usepackage{amsfonts}
\usepackage{amssymb}
\usepackage{graphicx}
\usepackage{amsthm}
\usepackage[toc,page]{appendix}
\usepackage{geometry}
\usepackage[all]{xy}
\usepackage{lmodern}
\usepackage[T1]{fontenc}
\usepackage[dvipsnames]{xcolor}

\usepackage{graphics}
\usepackage{comment}
\usepackage{csquotes}
\usepackage{hyperref}
\usepackage{enumerate}
\usepackage{caption}
\usepackage{subcaption}
\usepackage{float}
\usepackage{multicol}
\usepackage{algpseudocodex}
\usepackage{algorithm}

\usepackage[normalem]{ulem}

\numberwithin{equation}{section}

\newtheorem{theorem}{Theorem}[section]	
\newtheorem*{theorem*}{theorem}
\newtheorem{proposition}[theorem]{Proposition} 
\newtheorem{lemma}[theorem]{Lemma}
	
\newtheorem{definition}[theorem]{Definition}
\newtheorem*{def*}{Definition}
\newtheorem{cor}[theorem]{Corollary}
\newtheorem{remark}[theorem]{Remark}
\newtheorem{example}[theorem]{Example}

\usepackage{tikz}
\usetikzlibrary{decorations,arrows}
\usetikzlibrary{decorations.pathmorphing}
\usepgflibrary{decorations.pathreplacing} 
\usepackage{pgf}
\usetikzlibrary{patterns}

\tikzset{
	schraffiert/.style={pattern=horizontal lines,pattern color=#1},
	schraffiert/.default=black
}

\tikzset{
	ultra thin/.style= {line width=0.1pt},
	very thin/.style=  {line width=0.2pt},
	thin/.style=       {line width=0.4pt},
	semithick/.style=  {line width=0.6pt},
	thick/.style=      {line width=0.8pt},
	very thick/.style= {line width=1.2pt},
	ultra thick/.style={line width=2.4pt}
}

\newcommand{\C}{{\mathbb C}}

\newcommand{\N}{{\mathbb N}}

\newcommand{\R}{{\mathbb R}}

\newcommand{\Z}{{\mathbb Z}}

\newcommand{\Aa}{{\mathcal A}}
\newcommand{\Ee}{{\mathcal E}}
\newcommand{\Dd}{{\mathcal D}}
\newcommand{\Jj}{{\mathcal J}}
\newcommand{\Kk}{{\mathcal K}}

\newcommand{\Ss}{{\mathcal S}}
\newcommand{\Vv}{{\mathcal V}}

\newcommand{\Orb}{{\mathrm Orb}}				
\DeclareMathOperator{\dist}{{\mathrm dist}}     
\newcommand{\dc}{{\mathrm d_c}}					

\newcommand{\PatstarZd}{\mathrm{Pat}(\Aa^{\Z^d})}			
\newcommand{\Patstar}{\mathrm{Pat}(\Aa^\Gamma)}	
\newcommand{\Stab}{\mathrm{Stab}}				
\newcommand{\CLR}{C_{\mathrm{LR}}}				

\newcommand{\HeiR}{H_3(\mathbb{R})}		
\newcommand{\HeiZ}{H_3(2\mathbb{Z})}		

\newcommand{\mv}{\vec{m}}
\newcommand{\Kmv}{K_{\vec{m}}}

\newcommand{\lodim}{\underline{\dim}_B}

\newcommand{\Gtable}{G_{\textit{table}}}

\definecolor{grau}{gray}{0.9} 
\definecolor{gelb}{RGB}{225,225,0} 
\definecolor{blau}{RGB}{0,0,255} 
\definecolor{rot}{RGB}{255,0,0} 

\usetikzlibrary{shapes}
\usetikzlibrary{positioning}

\title{Spectral approximation for substitution systems}
\author{Ram Band, Siegfried Beckus, Felix Pogorzelski and Lior Tenenbaum}

\address{Department of Mathematics\\
Technion - Israel Institute of Technology\\
Haifa, Israel\\
and\\
Institute of Mathematics\\
University of Potsdam\\
Potsdam, Germany}
\email{ramband@technion.ac.il}

\address{Institute of Mathematics\\
University of Potsdam\\
Potsdam, Germany}
\email{beckus@uni-potsdam.de}

\address{Department of Mathematics and Computer Science\\
University of Leipzig, Leipzig, Germany}
\email{felix.pogorzelski@math.uni-leipzig.de}

\address{Department of Mathematics\\
 Technion - Israel Institute of Technology\\
 Haifa, Israel}
\email{tenen25@campus.technion.ac.il}

\begin{document}
	\begin{abstract}
		We study periodic approximations of aperiodic Schrödinger operators on lattices in Lie groups with dilation structure. The potentials arise through symbolic substitution systems that have been recently introduced in this setting. 
		We characterize convergence of spectra of associated Schrödinger operators in the Hausdorff distance via properties of finite graphs. 
		As a consequence, new examples of periodic approximations are obtained. 
		We further prove that there are substitution systems that do not admit periodic approximations in higher dimensions, in contrast to the one-dimensional case. 
		On the other hand, if the spectra converge, then we show that the rate of convergence is necessarily exponentially fast. 
		These results are new even for substitutions over $\mathbb{Z}^d$. 
	\end{abstract}
	
	\maketitle

\tableofcontents
	
\section{Introduction and guiding example}

For a finite set $\mathcal{A} \subseteq \mathbb{R}$ and a function $V:\Z^d\to\mathcal{A}$, consider the self-adjoint and bounded operator $H_V:\ell^2(\Z^d)\to\ell^2(\Z^d)$ defined by
\begin{equation}\label{eq:Hamiltonian_Zd}
(H_{V}\psi)(\gamma) := \sum_{\eta\in\Z^d:\, |\eta-\gamma|= 1} \psi(\eta) + V(\gamma) \psi(\gamma), \qquad \psi\in\ell^2(\Z^d), \gamma\in \Z^d.
\end{equation}

Though determining the spectrum $\sigma(H_V)$ and its properties is very hard in general, the precise structure of $V$ may  allow one to get some insights into the spectral theory of $H_V$. For instance, such potentials can be found in the theory of aperiodic order. For background information on aperiodic order and its interrelations with spectral theory, we refer to the monographs \cite{Queff10,BaakeGrimm13}. 
In the one-dimensional case, there are various powerful methods to explore the spectral properties of such operators, depending on the nature of the potential, see e.g.\@ \cite{Teschl_Jacobi,DaEmGo15-survey,Dam17-Survey,Jitomirskaya2019,DF22,DaFi24-book_2}. 
Generally, less techniques are available in higher dimensions. There are some explicit results
for the combinatorial Laplacian or the adjacency operator, see for instance the survey \cite{MohWoe89}, \cite{BVZ97, KS99} for  Heisenberg Cayley graphs, \cite{GZ01,BW05, GLN16,GS21} for the Lamplighter group and related graphs, or \cite{GLN18, GNP21} for certain Schreier graphs. In addition, interesting developments were achieved for trees \cite{KeLeWa12,KeLeWa13,KeLeWa15,AvBreuSim_PerTree20}.
However, the picture becomes much less clear for aperiodic operators in higher dimensions. An approach in this direction is to deal with dynamically-defined potentials \cite{Dam17-Survey,BBdN18} and to find suitable periodic approximations. 
The latter are of particular interest since they can be studied via Floquet-Bloch theory \cite{ReedSimonIV,Kuch16}. Here, $V$ is called periodic if there are only finitely many translations of $V$ by elements in $\Z^d$. For instance, they can be used to numerically compute the spectrum \cite{BBdN18}, detect spectral gaps \cite{ChElYu90,HeMoTe24}, prove convergence of the Lebesgue measures \cite{AvMoSi91,Last94}, estimate fractal dimensions of the spectrum by finding suitable covers \cite{LiuWen04,DamTch07,DamEmbGorTch08,DaGoYe16} or to study whether all possible spectral gaps are there \cite{Raym95,DaGoYe16,BBL24,DaEmFi24,BaBeBiTh25}.

\medskip

The theme of this work is to address the issue whether the spectra as sets converge in the Hausdorff metric, cf.\@ \cite{Ell82,BIT91,BBdN18}. If so, one might also attempt to  estimate the rate of convergence, see e.g. \cite{ChElYu90,AvMoSi91,Bel94,CoPu12,BBC19,BecTak21}. We answer these questions affirmatively in this work for specific potentials defined by substitutions.
The latter are local inflation rules \cite{Queff10,BaakeGrimm13} describing the potential $V$ and are of particular interest in the areas of aperiodic order and symbolic dynamics. 
We adopt the framework from \cite{BHP21-Symbolic}, where substitution systems are defined via a lattice $\Gamma$ in a Lie group with dilation structure, generalizing the classical set-ups over $\Z$ and $\Z^d$. For a given potential $V$, we construct a sequence $(V_n)$ of potentials by iterating the substitution rule on an initial configuration. The main results of the present paper can be summarized as follows.

\begin{itemize}
\item We characterize the convergence of  $\sigma(H_{V_n})$ towards $\sigma(H_V)$ with respect to the Hausdorff metric, see Theorem~\ref{thm:charact_convergence}. In particular, we give a verifiable criterion in terms of finite graphs associated with the substitution. The construction of these graphs  depends on a geometric notion that will be called {\em testing tuples}, cf.\@ Definition~\ref{def:TestingTuple}. 

\item We show that if $\sigma(H_{V_n})$ converges to $\sigma(H_V)$, then the Hausdorff distance of the spectra necessarily decays exponentially fast, see Corollary~\ref{cor:SpectralEstimates} as a consequence of Theorem~\ref{thm:upper-bound-estim}.

\item We investigate the conditions for convergence and compute testing tuples 
in concrete examples, see Section~\ref{subsec:Block_Substitutions} for block subsitutions over $\Gamma = \Z^d$ and Section~\ref{subsec:Heisenberg_Example} for the discrete Heisenberg group $\Gamma = H_3(2\Z)$. For the Heisenberg group, we also give explicit examples of aperiodic configurations admitting periodic approximations.

\item We show that for $\Gamma = \Z^2$ not all periodic initial configurations are suitable for approximation, see Corollary~\ref{cor:TableTiling_Letter_BadAppr}, and that periodic approximations might not exist at all for certain substitutions, see Corollary~\ref{cor:PrimSubst_NotPerAppr}.  
\end{itemize} 

If $\Gamma=\Z$, the methods of this paper can also be applied to substitution systems which are not of constant length. This is worked out in \cite{Ten24}. For Sturmian Hamiltonians, explicit spectral estimates have been recently proven in \cite{BBT24}. For rotation numbers with eventually periodic continued fraction expansion, the underlying dynamical system comes from a substitution. In these cases, \cite{BBT24} provides the analogous results for Corollary~\ref{cor:SpectralEstimates}.

\medskip

The paper is organized as follows. We introduce some preliminaries in Section~\ref{sec:system} and illustrate our results for the guiding example of the table tiling substitution in Section~\ref{subsec:tableTiling}. Section~\ref{sec:BeyondAbelian} is devoted to the presentation of the main results of our work. The framework for symbolic substitution systems, along with some results needed for our purposes are introduced in Section~\ref{subsec:Substitutions_general}. We define the concepts of substitution graphs  and testing domain in Section~\ref{subsec:TestingDomains}. These objects play a fundamental role in our first main theorem, Theorem~\ref{thm:charact_convergence} on the characterization of convergence, stated in Section~\ref{subsec:DynSyst_SpectralConv}. As a consequence, one also obtains existence of aperiodic substitution systems that are not periodically approximable, cf.\@ Corollary~\ref{cor:PrimSubst_NotPerAppr}.
 The second main theorem, Theorem~\ref{thm:upper-bound-estim}, concerning convergence with exponential speed, is formulated in Section~\ref{subsec:MainResults}. 
We derive in Corollary~\ref{cor:SpectralEstimates} exponential convergence of the spectra with respect to the Hausdorff distance. Section~\ref{sec:Block_and_Heisenberg} is devoted to the construction of testing tuples for explicit examples. This concerns block substitutions for $\Gamma = \Z^d$ (see Proposition~\ref{prop:test-euclid} in  Section~\ref{subsec:Block_Substitutions})  and substitution systems over the discrete Heisenberg group $\Gamma= H_3(2\Z)$ (see Proposition~\ref{prop:Periodic_Heisenberg} in Section~\ref{subsec:Heisenberg_Example}). In Section~\ref{sec:Cond-sect} we give the proof of Theorem~\ref{thm:charact_convergence}, our first main theorem. The following Section~\ref{sec:Rate_Converg} is concerned with the proof of Theorem~\ref{thm:upper-bound-estim}, our second main theorem. In Section~\ref{subsec:LowerBounds} we additionally obtain a lower bound on the convergence rate, see Proposition~\ref{prop:lower-bound-estim}. In the final Section~\ref{sec:Testing_Tuple_reduction}, we provide an algorithm, along with its mathematical foundation, to reduce the size of giving testing domains. Using computer assistance, we obtain a rather small testing domain for the Heisenberg group, see Proposition~\ref{prop:Heisen-test-dom}.

\subsection*{Acknowledgements}
L.T.\@ wishes to thank Alan Lew and Philipp Bartmann for insightful discussions. 
We wish to thank Daniel Lenz for pointing out his work on densely repetitive Delone sets \cite{Lenz04}, which was helpful for deriving the bounds in Section~\ref{subsec:LowerBounds}. 
The authors are grateful to Pascal Vanier for discussions on the results in \cite{DuLeSh05,Oll08,JeVa20,Ballier23}. 
We also wish to thank Yannik Thomas for useful comments on an earlier version. 
This work was partially supported by the Deutsche Forschungsgemeinschaft [BE 6789/1-1 to S.B.] and [PO 2383/2-1 to F.P.]. R.B.\@ and L.T.\@ were supported by the Israel Science Foundation (ISF Grant No. 844/19). We are grateful for the hospitality and the excellent working conditions provided by the University of Leipzig, University of Potsdam and the Technion during mutual visits.

\subsection{The underlying dynamical system} \label{sec:system}
Let $\Gamma$ be  a countable discrete group. Important examples are $\Gamma=\Z^d$ or the discrete Heisenberg group $\Gamma= \HeiZ$. Suppose that the potential $V:\Gamma\to\R$ takes only finitely many values. Then the finite set $\Aa:= \{V(\gamma)\,|\, \gamma\in\Gamma\}$ is called the \emph{alphabet}. The potential $V$ can be seen as an element of the product space $\Aa^\Gamma= \big\{ \omega:\Gamma\to \Aa \big\}$, which we call the {\em configuration space}. A specific metric on $\Aa^\Gamma$ is defined as follows: suppose that there is a left-invariant metric $d_\Gamma$ on $\Gamma$, i.e. \@ $d_\Gamma(\gamma\eta_1,\gamma\eta_2)=d_\Gamma(\eta_1,\eta_2)$ for all $\gamma,\eta_1,\eta_2\in\Gamma$. For instance, $\Gamma=\mathbb{Z}^d$ can be equipped with the Euclidean metric, or the Heisenberg group $\Gamma=H_3(\mathbb{Z})$ can be equipped with the metric inherited from the Cygan-Kor{\'a}nyi norm, see equation \eqref{eq:Heisenberg_CyganKoranyiNorm} below.  Then
\begin{equation} \label{eq:config-metric}
\dc(\omega,\rho):= 
	\inf \Big\{ \frac{1}{r+1}  \Big\vert\; r\geq 0 \; \textrm{such that} \; \omega\vert_{B(e,r)}= \rho \vert_{B(e,r)} \Big \},
	\qquad \omega,\rho\in\Aa^\Gamma,
\end{equation}
defines an ultra metric on the configuration space $\Aa^\Gamma$. Here $B(e,r):=\{\gamma\in\Gamma\,|\, d_\Gamma(\gamma,e)<r\}$ denotes the open ball of radius $r$ in $\Gamma$ around the neutral element $e\in\Gamma$.
The group $\Gamma$ acts continuously on $\Aa^\Gamma$ via left translations
\[
	(\gamma \omega)(\eta):= \omega (\gamma^{-1}\eta) \quad \text{for all} \quad \gamma,\eta \in \Gamma.	
\]
Thus, $(\Aa^\Gamma,\Gamma)$ defines, via the previously defined action, a dynamical system. 
A nonempty set $\Omega\subseteq \Aa^\Gamma$ is called \emph{invariant} if $\gamma \Omega \subseteq \Omega$ for all $\gamma \in \Gamma$. Then the \emph{space of subshifts} is defined by
	\[ 
	\Jj := 
		\Big\{  \Omega\subseteq \Aa^\Gamma \vert\; \Omega \textrm{ is invariant, closed, nonempty} \Big \}.
	\]
Particular elements of interest in $\Jj$ are the orbit closures $\overline{\Orb(\omega)}$ for $\omega\in\Aa^\Gamma$ where $\Orb(\omega):=\{\gamma\omega\,|\, \gamma\in\Gamma\}$. We are mainly interested in the subshift generated by the potential $V$, namely $\overline{\Orb(V)}\in\Jj$. The set $\Jj$ is naturally equipped with a metric -- the Hausdorff metric inherited from $\dc$ -- that encodes spectral properties \cite{BBdN18,BBC19,BP18,BecTak21}.

\medskip

We first recall the definition of the Hausdorff distance. Let $(X,d)$ be a metric space and $\Kk(X)$ be the set of all nonempty compact subsets of $X$. Then, the \emph{Hausdorff metric} between two  $A,B\in \Kk(X)$ is defined by
\begin{equation*}
	d_H(A,B)= \max \Big\{ \underset{a\in A}{\sup}\; \dist(a,B), \underset{b\in B}{\sup}\; \dist(b,A) \Big\},
\end{equation*}
where $\dist(a,B) := \inf_{b^{\prime} \in B} d(a,b^{\prime})$. By definition, the Hausdorff metric is induced from the underlying metric $d$ on $X$. We are mainly interested in the case
\begin{itemize}
\item $X=\R$ with Hausdorff metric $d_H$ on $\Kk(\R)$ induced by the Euclidean distance $|\cdot|$; and
\item $X=\Aa^\Gamma$ with Hausdorff metric $\delta_H$ induced by the metric $\dc$ and its restriction to $\Jj\subseteq \Kk(\Aa^\Gamma)$.
\end{itemize}
Note that $(\Jj,\delta_H)$ is itself a compact metric space since $\Jj\subseteq \Kk(\Aa^\Gamma)$ is closed and $\Kk(\Aa^\Gamma)$ is compact\color{black}, see \cite[Proposition~4]{BBdN18}. The Hausdorff metric $\delta_H$ on $\Jj$ can be expressed by local patches of configurations (see equation~\eqref{eq:HausdorffMetric_Dictionaries} below), which we briefly introduce.

For $M\subseteq \Gamma$, a map $P:M\to\Aa$ is called a {\em patch} with {\em support} $M$. If $M$ is finite, we say $P$ is a finite patch. The \emph{set of all patches} is denoted by $\Patstar$. Note that $\Patstar$ includes the empty patch $\text{\o}\in\Aa^\emptyset$ and infinite patches such as elements of $\Aa^\Gamma$. 
A patch $P\in\Aa^M$ is a {\em subpatch} of $Q\in\Aa^K$ ($P\prec Q$) if there is a $\gamma\in\Gamma$ such that $\gamma M\subseteq K$ and $Q(\gamma \eta) = P(\eta)$ for all $\eta\in M$. 
By convention, the empty patch $\text{\o}\in\Aa^\emptyset$ is a subpatch of any configuration $\omega\in\Aa^\Gamma$. 
Clearly, the relation $\prec$ defined on patches is  a partial order. 
For $\omega\in\Aa^\Gamma$, the associated dictionary $W(\omega)$ is the set of all finite subpatches of $\omega\in\Aa^\Gamma$. 
Similarly, one defines $W(\Omega):=\bigcup_{\omega\in\Omega} W(\omega)$ for a subshift $\Omega\in\Jj$. For a finite set $M\subseteq \Gamma$, define $W(\omega)_M:=W(\omega)\cap\Aa^M$ for $\omega\in\Aa^\Gamma$ and $W(\Omega)_M:=W(\Omega)\cap\Aa^M$ for $\Omega\in\Jj$. With this at hand, we get
\begin{equation}
\label{eq:HausdorffMetric_Dictionaries}
\delta_H(\Omega_1,\Omega_2)
	= \inf\left\{\left. \frac{1}{r+1} \,\right|\, r\geq 0 \; \textrm{such that} \;   W(\Omega_1)_{B(e,r)}=W(\Omega_2)_{B(e,r)} \right\}, 
	\qquad \Omega_1,\Omega_2\in\Jj,
\end{equation}
from our choice of the $\dc$ on $\Aa^\Gamma$ and the definition of the Hausdorff metric. Note that by convention $W(\Omega)_{B(e,0)}=\{\text{\o}\}$.

A {\em Schrödinger operator $H$ with finite range} is a family of self-adjoint operators $H_\omega:\ell^2(\Gamma)\to\ell^2(\Gamma)$ for $\omega\in\Aa^\Gamma$ defined by 
\begin{equation}\label{eq:Hamiltonian}
(H_\omega\psi)(\gamma) := \sum_{\eta\in B} t_\eta\big(\gamma^{-1}\omega\big) \psi(\gamma\eta), 
	\qquad \psi\in\ell^2(\Gamma), \gamma\in \Gamma,
\end{equation}
where $B\subseteq \Gamma$ is finite with $B=B^{-1}$ while $t_\eta:\Aa^\Gamma\to\C, \eta\in B,$ are continuous and satisfy that $t_{\eta}(\omega)=\overline{t_{\eta^{-1}}(\eta \omega)}$, where $\overline{z}$ is the complex conjugate of $z$. The latter conditions guarantee that $H_\omega$ is self-adjoint for each $\omega\in\Aa^\Gamma$. 
Thus, the spectrum $\sigma(H_\omega)$ of $H_\omega$ is a compact subset of $\R$.
We say $H$ is {\em strongly pattern equivariant} if the coefficients $t_\eta:\Aa^\Gamma\to\C, \eta\in B,$ are continuous and take finitely many values.
These conditions are equivalent to $t_\eta$ being locally constant. Thus, the value $t_\eta(\omega)$ depends only on the patch $\omega|_{B(e,r)}$ for suitable $r>0$.
The operator defined in \eqref{eq:Hamiltonian_Zd} is a strongly pattern equivariant Schrödinger operator with finite range using the viewpoint that $V\in\Aa^\Gamma$.

Recently, it was shown that the convergence of dynamical systems in $(\Jj,\delta_H)$ is tightly connected to the convergence of the spectrum of such operators in $(\Kk(\R),d_H)$ \cite{BBdN18,BBC19,BecTak21} and other spectral quantities \cite{BP18} if $\Gamma$ is amenable. 
This is the starting point of this work. 

Let $\Omega\in\Jj$. To study the spectral properties of Schrödinger operators associated with $\omega\in\Omega$, one aims to find periodic approximations. 
Here, a configuration $\omega_0\in\Aa^{\Gamma}$ is {\em periodic}, if its orbit $\Orb(\omega_0)$ is finite. Then $\Omega$ is {\em periodically approximable} if there is a sequence of periodic $\omega_n\in\Aa^{\Gamma}$ such that $\lim_{n\to\infty} \delta_H\big(\Orb(\omega_n),\Omega\big)=0$. In combination with the results obtained in \cite{BBdN18,BBC19,BecTak21}, we obtain periodic approximations of the associated Schrödinger operators and their spectra. Thus, the question arises when a subshift is periodically approximable.

Our main results hold for a quite general setup. Concretely, we consider subshifts defined via substitutions on lattices in homogeneous Lie groups that were recently introduced \cite{BHP21-Symbolic}. Since the formal statements need some more introduction, we start presenting the main results along a guiding example, namely the table tiling substitution. Later we also provide an example on a non-abelian group $\Gamma$ -- the discrete Heisenberg group, see Section~\ref{subsec:Heisenberg_Example}.

\subsection{A guiding example: The table tiling substitution}
\label{subsec:tableTiling}

Let $\Aa=\{$
\hspace*{-0.2cm}
\begin{tikzpicture}
	\filldraw [rot] (0,0.5) circle (2.5pt);
\end{tikzpicture}
\hspace*{-0.2cm} , \hspace*{-0.2cm}
\begin{tikzpicture}
	\filldraw [gelb] (0.35,0.5) circle (2.5pt);
\end{tikzpicture}
\hspace*{-0.1cm}, \hspace*{-0.2cm}
\begin{tikzpicture}
	\filldraw [blau] (0.7,0.5) circle (2.5pt);
\end{tikzpicture}
\hspace*{-0.1cm}, \hspace*{-0.2cm}
\begin{tikzpicture}
	\draw [thin, fill=grau]  (1.05,0.5) circle (2.5pt);
\end{tikzpicture} \hspace*{-0.2cm}
$\}$, 
$\Gamma=\Z^2$ and $K=\{-1,0\}^2$. The {\em table tiling substitution} is defined by a substitution rule $S_0:\Aa\to\Aa^{K}$ where
\vspace*{-1cm}
\begin{figure}[H] 
	\begin{center}		\includegraphics[scale=3]{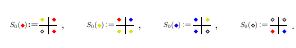}
	\end{center}
\end{figure}
\vspace*{-1cm}

It is standard to extend a substitution rule $S_0$ to a substitution map $S:\mathrm{Pat}(\Aa^{\Z^2})\to\mathrm{Pat}(\Aa^{\Z^2})$ by acting letter wise, see e.g. \cite[Chapter 5.1]{Queff10} and \cite[Chapter 4.9]{BaakeGrimm13}. Note that the restriction $S:\Aa^{\Z^2}\to\Aa^{\Z^2}$ is continuous. For convenience of the reader, the concrete mathematical statement is provided in a more general setting in Proposition~\ref{prop:SubstitutionMap} originating from \cite[Proposition~2.7]{BHP21-Symbolic}. For now, the idea of extending the substitution rule is sketched in Figure~\ref{fig:TableTiling_Iterations}.

\begin{figure}[htb] 
\begin{center}
\includegraphics[scale=2.2]{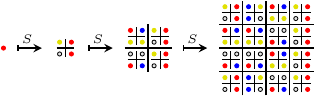}
\caption{The first three iterations of the table tiling substitution starting from the letter $\textcolor{red}{\Huge\bullet}$.}
\label{fig:TableTiling_Iterations}
\end{center}
\end{figure}

A patch $P\in\Aa^M$ is called {\em $S$-legal} (with respect to the table tiling substitution) if there exists a letter $a\in\Aa$ and an integer $n\in\N$ such that $P\prec S^n(a)$. The set of all $S$-legal patches is denoted by $W(S)$, which defines a dictionary. For instance, the $S$-legal patches with support $T=\{0,1\}^2$ are given in Figure~\ref{fig:TableTiling_T-legal_Patch}, see e.g. \cite[Remark 4.17]{BaakeGrimm13}.

\begin{figure}[hb] 
\begin{center}
\includegraphics[scale=3.3]{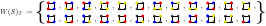}
\caption{The $S$-legal patches for the table tiling substitution with support $T=\{0,1\}^2$.}
\label{fig:TableTiling_T-legal_Patch}
\end{center}
\end{figure}

The associated subshift of the table tiling substitution is given by 
$$
\Omega(S):=\big\{\omega\in\Aa^{\Z^2} \,|\, W(\omega)\subseteq W(S)\big\} \in\Jj.
$$

The typical approach to construct periodic approximations for $\Omega(S)$ is to take a configuration $\omega_a$, that is constantly equal to a fixed letter in $a\in\Aa$ and apply the substitution $S^n(\omega_a)$. 
While this approach works in the typical cases (but not all see \cite[Section 2]{Ten24}) if $\Gamma=\Z$, it fails for any letter in our example, see Corollary~\ref{cor:TableTiling_Letter_BadAppr} below. 

In order to resolve this issue, we show that the convergence can be checked by studying a finite graph associated with the substitution.
In order to do so, let us shortly recall some basic graph theoretic notions.
We call $G=(\Vv,\Ee)$ a {\em directed graph} with {\em vertex set} $\Vv$ and {\em edge set} $\Ee\subseteq \Vv\times\Vv$, if $\Vv$ is a finite (nonempty) set. A tuple $(v,w)\in\Ee$ is a directed edge from $v$ to $w$. A \emph{(directed) path} of length $\ell\in\N$ in $G$ is a finite tuple $(v_0,v_1,...,v_\ell)\in \Vv^{\ell+1}$, such that $(v_j,v_{j+1})\in\Ee$ for every $0\leq j \leq \ell-1$. A path $(v_0,...,v_\ell)$ in $G$ is called a \emph{closed path} if $v_0=v_\ell$. 
A path $(v_0,...,v_\ell)$ in $G$ is called a \emph{subpath} of $(u_0,...,u_{\tilde{\ell}})$, if $\ell\leq \tilde{\ell}$ and there exists an $0\leq i \leq \tilde{\ell}$ such that $v_j=u_{i+j}$ for every $0\leq j\leq \ell$.

For $T:=\{0,1\}^2$, we define the {\em graph $\Gtable$ associated with the table tiling substitution} by the directed graph with vertex set $\Vv=\Aa^T$ and the edge set $\Ee$ defined by
$$
(P,Q)\in \Ee \quad :\Longleftrightarrow\quad Q\prec S(P) \textrm{ and } P,Q \not\in W(S).
$$ 
We note that while the graph is finite its vertex set has $|\Aa^T|=256$ vertices where $|\Aa^T|$ denotes the cardinality of the set $\Aa^T$.
A sketch of a subgraph of $\Gtable$ is provided in Figure~\ref{fig:TableTiling_Graph}.

\begin{figure}[hbt] 
\begin{center}
\includegraphics[scale=0.96]{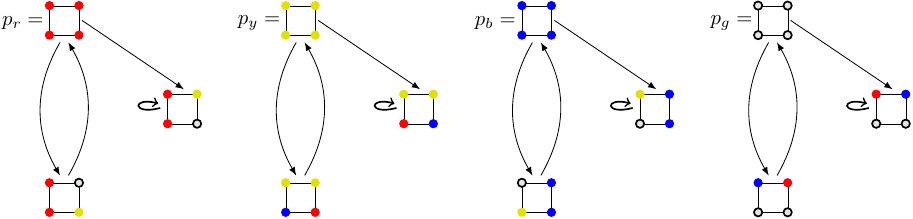}
\caption{ Subgraphs of the graph $\Gtable$ are plotted.}
\label{fig:TableTiling_Graph}
\end{center}
\end{figure}

It is worth emphasizing that we only have edges between patches that are not $S$-legal (illegal). Therefore this graph tracks if an illegal patch of size $T$ persists under the iteration of the substitution or not. A common strategy to periodically approximate a substitution subshift $\Omega(S)$ is by choosing a finite $S$-legal patch and extend it via periodic boundary conditions to obtain $\omega_0\in\Aa^{\Z^2}$. When gluing the patches periodically together, we may create illegal patches at the boundary in general. In order to deal with this, we need to check if such errors (illegal patches) eventually disappear when iterating the substitution.

Since we only have finitely many patches of size $T$, one gets a cycle in $\Gtable$, whenever an illegal patch persists under the iteration of the substitution. 
On the other hand, if the paths starting in patches of $\omega_0\in\Aa^{\Z^2}$ are all finite, then $S^n(\omega_0)$ contains only legal patches of size $T$ for $n$ large enough. 
Since $T=\{0,1\}^2$ is chosen to be large enough as well (formally introduced as a testing domain below), we show that for each radius $R>0$, the patches with support $B(e,R)$ in $S^n(\omega_0)$ are $S$-legal for $n$ large enough. Since the table tiling substitution is also primitive, all patches eventually appear. Thus, the orbit closure of $S^n(\omega_0)$ is converging to $\Omega(S)$ using the description of the Hausdorff metric in equation~\eqref{eq:HausdorffMetric_Dictionaries}.
Conversely, if a patch of size $T$ is contained in a cycle of $\Gtable$, then all iterations $S^n(\omega_0)$ would contain at least one illegal patch of size $T$. Thus, the subshifts do not converge using equation~\eqref{eq:HausdorffMetric_Dictionaries}.
The connection to the spectral convergence and the convergence of the subshifts follows then from \cite{BBdN18}.

With this at hand, we obtain the following special case of our main result (Theorem~\ref{thm:charact_convergence}).

\begin{proposition} \label{prop:TableTiling_charact_convergence}
The following assertions are equivalent for the table tiling substitution and $\omega_0\in\Aa^{\Z^2}$ and $T:=\{0,1\}^2$.
\begin{enumerate}[(i)]
	\item For all Schrödinger operators $H$ with finite range, we have
$$
\lim\limits_{n\to\infty}\sigma(H_{S^n(\omega_0)})=\sigma(H_{\omega}),\qquad \omega\in\Omega(S),
$$
	\item  We have $\lim\limits_{n\to\infty}\delta_H\big(\overline{\Orb(S^n(\omega_0))},\Omega(S)\big)=0$.
	\item  Each directed path in $\Gtable$, starting in a vertex of $W(\omega_0)_T\subseteq \Aa^T$, does not contain a closed subpath.
\end{enumerate}
In particular, if $W(\omega_0)_T\subseteq W(S)$, then these equivalent conditions are satisfied.
\end{proposition}

\begin{proof}
Due to Proposition~\ref{prop:blockSubst_DilationDatum}, the table tiling substitution falls in our general setting of substitutions defined in Section~\ref{sec:BeyondAbelian} and it is elementary to check that the table tiling substitution is primitive (Definition~\ref{def:primitive}). Then the statement follows from the main Theorem~\ref{thm:charact_convergence} in combintation with Proposition~\ref{prop:test-euclid}.
\end{proof}

\begin{cor}
\label{cor:TableTiling_Letter_BadAppr}
Let $S$ be the table tiling substitution map. For $a\in\Aa$, define $\omega_a\in\Aa^\Gamma$ by $\omega_a(\gamma)=a$ for all $\gamma\in\Z^2$. Then 
$$
\lim_{n\to\infty} \Orb(S^n(\omega_a)) \neq \Omega(S),\qquad a\in\Aa.
$$
In particular, for each $a\in\Aa$, there is a Schrödinger operator $H$ with finite range such that
$$
\lim\limits_{n\to\infty}\sigma(H_{S^n(\omega_a)}) \neq \sigma(H_{\omega}),\qquad \omega\in\Omega(S).
$$
\end{cor}

\begin{proof}
Define $p_a\in\Aa^T$ to be $p_a(\gamma)=a$ for $\gamma\in T=\{0,1\}^2$.
The vertex $p_a\in W(\omega_a)_T$ is contained in a closed subpath in $\Gtable$, see Figure~\ref{fig:TableTiling_Graph}. Thus, Proposition~\ref{prop:TableTiling_charact_convergence}~(iii) does not hold, proving the statement.
\end{proof}

The operators whose spectrum is not convergent are constructed by choosing a large potential on the illegal patches.

On the positive side, Proposition~\ref{prop:TableTiling_charact_convergence} allows us also to find periodic approximations for the table tiling. Define $\omega_{rb},\omega_{gy}\in\Aa^{\Z^2}$, by
$$
\omega_{rb}(\gamma) 
	:= \begin{cases}
		\begin{tikzpicture}
			\filldraw [rot] (0,0.5) circle (2.5pt);
		\end{tikzpicture},\qquad \gamma\in (2\Z)^2 \cup (1,1)+(2\Z)^2,\\
		\begin{tikzpicture}
			\filldraw [blau] (0.7,0.5) circle (2.5pt);
		\end{tikzpicture},\qquad \textrm{else},
	\end{cases},
\qquad
\omega_{gy}(\gamma) 
	:= \begin{cases}
		\begin{tikzpicture}
			\draw [thin, fill=grau]  (1.05,0.5) circle (2.5pt);
		\end{tikzpicture} ,\qquad \gamma\in (2\Z)^2 \cup (1,1)+(2\Z)^2,\\
		\begin{tikzpicture}
			\filldraw [gelb] (0.35,0.5) circle (2.5pt);
		\end{tikzpicture},\qquad \textrm{else},
	\end{cases}
$$
for $\gamma\in\Z^2$, see a sketch in Figure~\ref{fig:TableTiling_PeriodicAppr}. Clearly, $\Orb(\omega_{rb})$ and $\Orb(\omega_{gy})$ contain only two different elements, namely they are periodic.
With this at hand, we obtain the following statement using that $\Omega(S)$ for the table tiling substitution $S$ is linearly repetitive with repetitivity constant $\CLR>0$, see Definition~\ref{def:lin-rep} and Theorem~\ref{thm:Omega(S)}.

\begin{proposition}
\label{prop:Table_PeriodicAppr}
Let $S$ be the table tiling substitution. Then $\Omega(S)$ is periodically approximable and
$$
\lim_{n\to\infty} \Orb\big(S^n(\omega_{rb})\big) = \Omega(S) = \lim_{n\to\infty} \Orb\big(S^n(\omega_{gy})\big).
$$
Moreover, for each strongly pattern equivariant Schrödinger operator $H$ with finite range, we have for $i\in\{ rb, gy\}$,
$$
d_H\big( \sigma(H_{S^n(\omega_i)}), \sigma(H_{\omega}) \big) \leq \frac{\max\{2\CLR,4\}}{2^n} ,\qquad \omega\in\Omega(S), n\geq \frac{\log(2\CLR)}{\log(2)},
$$
where $\CLR>0$ is the linear repetitivity constant of the table tiling.
\end{proposition}

\begin{figure}[hbt] 
\begin{center}
\includegraphics[scale=2.65]{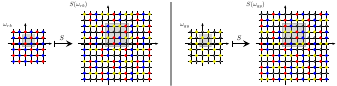}
\caption{The configurations $\omega_{rb},S(\omega_{rb}),\omega_{gy},S(\omega_{gy})\in\Aa^{\Z^2}$ are plotted. The gray shaded area indicated the block that is periodically extended.}
\label{fig:TableTiling_PeriodicAppr}
\end{center}
\end{figure}

\begin{remark}
Note that $S^n(\omega_{rb})$ and $S^n(\omega_{gy})$ are periodic by Proposition~\ref{prop:Periodic_Sn}. Moreover, $\Omega(S)$ is linearly repetitive (see Definition~\ref{def:lin-rep}) with linear repetitivity constant $\CLR>0$ since $S$ is primitive.
\end{remark}

\begin{proof}
Let $T=\{0,1\}^2$. Using Figure~\ref{fig:TableTiling_T-legal_Patch} and Figure~\ref{fig:TableTiling_PeriodicAppr}, we obtain $W(\omega_{rb})_T\subseteq W(S)$ and $W(\omega_{gy})_T\subseteq W(S)$.
Thus, $\lim_{n\to\infty} \Orb(S^n(\omega_i)) = \Omega(S)$ follows from Proposition~\ref{prop:TableTiling_charact_convergence} for $i\in\{ rb, gy\}$. Namely, $\Omega(S)$ is periodically approximable since $S^n(\omega_i)$ is periodic by Proposition~\ref{prop:Periodic_Sn}. 
The quantitative estimates are now a consequence of Corollary~\ref{cor:RateConv_Zd}.
\end{proof}

These specific periodic approximations were already found in the thesis \cite[Section 7.7]{BeckusThesis}. However, the spectral estimates are new, as well as the fact that all single letter approximations fail to converge. 
Like Proposition~\ref{prop:TableTiling_charact_convergence}, these spectral estimates are a special case of our second main result Theorem~\ref{thm:upper-bound-estim}, which holds in the realm of substitutions on lattices of homogeneous Lie groups.

\section{Main results for general substitution systems} 
\label{sec:BeyondAbelian}

In this section, we introduce the class of substitutions studied in this work following \cite{BHP21-Symbolic}. 
In addition, we present our main results proven in the subsequent sections and their consequences.

\subsection{Symbolic substitution systems}
\label{subsec:Substitutions_general}

Following \cite{BHP21-Symbolic}, one needs some geometric data -- a so-called dilation datum -- and combinatorial data -- a so-called substitution datum --  to describe a substitution on $\Aa^\Gamma$.
For convenience of the reader, we borrowed the notation of \cite{BHP21-Symbolic}.

A tuple $\Dd=\big( G,d, (D_\lambda)_{\lambda>0}, \Gamma , V  \big)$ is called a \emph{dilation datum}, if
\begin{enumerate}[(D1)]
	\item $G$ is a $1$-connected, locally compact, second countable, Hausdorff group and $d$ is a proper (i.e. closed balls with finite radius are compact), left-invariant metric on $G$ inducing the given topology on $G$;
	\item $(D_\lambda)_{\lambda>0}$ is a one-parameter group of automorphisms of $G$, called the underlying dilation family, such that
	$$
		d\big( D_\lambda(g), D_\lambda(h) \big) = \lambda \cdot d(g,h) \quad \text{for all} \quad \lambda>0, \; \text{and} \; g,h\in G; 
	$$
	\item $\Gamma< G$ is a uniform lattice satisfying that $D_\lambda[\Gamma] \subseteq \Gamma$ for some $\lambda > 1$ and $V$ is a Borel measurable, relatively compact  left-fundamental domain for $\Gamma$ such that $e$ is an element in the interior of $V$.
\end{enumerate}

Recall that a subgroup $\Gamma< G$ is called a \emph{uniform lattice} of a locally compact second countable Hausdorff group $G$, if $\Gamma$ is discrete and co-compact. A set $V\subseteq G$ is called a \emph{left-fundamental domain} for a uniform lattice $\Gamma$, if $G= \bigsqcup_{\gamma\in \Gamma} \gamma V$. Since we will only consider left-fundamental domains in this work, which in addition are Borel and relatively compact, we will refer to those simply as {\em fundamental domains}.

\medskip

For a fundamental domain $V$, $n\in\N$, $\lambda>1$ and $\emptyset\neq M\subseteq \Gamma$, we recursively define for $n\in\N$,
\begin{equation}
\label{eq:recursive_V(n)}
V_\lambda(0,M):=M\cdot V 
	\qquad \text{and} \qquad 
V_\lambda(n,M):=D_{\lambda}\Big[\big( V(n-1,M)\cap \Gamma \big)\cdot  V \Big] .
\end{equation}
Moreover, set $V_\lambda(n):=V_\lambda(n,\{e\})$ for $n\in\N$, where $e$ is the neutral element in $G$.
The intersection of $\Gamma$ with these sets are precisely the support of the iterates $S^n(P)$ of a patch $P\in\Aa^M$, see Proposition~\ref{prop:SubstitutionMap}~\eqref{enu:support_Sn_M} below. 

\medskip

Let $\Dd=(G,d,(D_\lambda)_{\lambda>0}, \Gamma,V)$ be a dilation datum. Then $\lambda_0>1$ is called {\em stretch factor associated with $\Dd$} if $D_{\lambda_0}[\Gamma]\subseteq \Gamma$ and $\lambda_0$ is sufficiently large relative to $V$. Here, $\lambda_0$ is called {\em sufficiently large relative to $V$} if there exist a constant $C_->0$, an integer $s\in\N_0:=\N \cup \{0\}$ and a $z\in\Gamma$ such that for all $n\in\N$,
$$
	D_{\lambda_0}^n\big[B\big(z,C_-)\big] \subseteq V_{\lambda_0}(s+n).
$$
If the underlying dilation datum $\Dd$ being referred to is clear, we will call $\lambda_0$ just a {\em stretch factor}. From now on, we use the notation $V(n,M)=V_{\lambda_0}(n,M)$ and $V(n)=V_{\lambda_0}(n)$ for $n\in\N_0$ and $M\subseteq \Gamma$ where $\lambda_0>1$ is a fixed stretch factor.

\medskip

A group $\Gamma$ that is part of a dilation datum as above is called a {\em homogeneous substitution lattice}.  Restricting the metric $d$ to $\Gamma \times \Gamma$ yields a left-invariant, proper metric $d_{\Gamma}$ on $\Gamma$. By \cite[Proposition~3.36]{BHP21-Symbolic}, every homogeneous substitution lattice has  \emph{exact polynomial growth} with respect to $d_{\Gamma}$, i.e.,\@  there exist constants $C>0$ and $\kappa>0$ such that
\[  
\lim_{r\to\infty}\frac{ |B^{\Gamma}(r)| }{Cr^\kappa} = 1,
\]
where $B^{\Gamma}(r) = \big\{ \gamma \in \Gamma:\, d_{\Gamma}(e,\gamma) \leq r \big\}$. The constant $\kappa$ only depends on $G$, and is called the \emph{homogeneous dimension}. As a consequence, every homogeneous substitution lattice $\Gamma$ is amenable, compare with \cite[Proposition~4.4.~(a)]{BHP20-LR}.

\medskip

A \emph{substitution datum} $\Ss:=(\Aa,\lambda_0, S_0)$ over a dilation datum $\Dd:=\big( G,d, (D_\lambda)_{\lambda>0}, \Gamma , V  \big)$ consists of
\begin{enumerate}[(S1)]
	\item a finite set $\Aa$ called the \emph{alphabet};
	\item a stretch factor $\lambda_0>1$ associated with $\Dd$;
	\item a map $S_0:\Aa\to \Aa^{D_{\lambda_0}[V]\cap \Gamma}$ called the \emph{substitution rule}.
\end{enumerate}

\begin{definition}
\label{def:substitution_general}
A {\em substitution on $\Aa^\Gamma$} is a dilation datum $\Dd$ together with a substitution datum $\Ss$. 
\end{definition}

\begin{example}
The table tiling substitution introduced in Section~\ref{subsec:tableTiling} is induced by the dilation $\Dd:=\big( G,d, (D_\lambda)_{\lambda>0}, \Gamma , V  \big)$ where 
$$
G=\R^2, \quad 
d(x,y):= \max_{1\leq j \leq 2} |x_j-y_j|, \quad 
D_\lambda(x_1,x_2\color{black})=(\lambda x_1,\lambda x_2\color{black}), \quad
\Gamma=\Z^2, \quad 
V=\big[-\tfrac{1}{2},\tfrac{1}{2}\big)^2
$$
and the substitution datum $\Ss:=(\Aa,\lambda_0, S_0)$ is
$$
\Aa=\{
\text{
\hspace*{-0.2cm}
\begin{tikzpicture}
	\filldraw [rot] (0,0.5) circle (2.5pt);
\end{tikzpicture}
\hspace*{-0.2cm} , \hspace*{-0.2cm}
\begin{tikzpicture}
	\filldraw [gelb] (0.35,0.5) circle (2.5pt);
\end{tikzpicture}
\hspace*{-0.1cm}, \hspace*{-0.2cm}
\begin{tikzpicture}
	\filldraw [blau] (0.7,0.5) circle (2.5pt);
\end{tikzpicture}
\hspace*{-0.1cm}, \hspace*{-0.2cm}
\begin{tikzpicture}
	\draw [thin, fill=grau]  (1.05,0.5) circle (2.5pt);
\end{tikzpicture} \hspace*{-0.2cm}
}
\}, \qquad 
\lambda_0=2,\qquad 
S_0 \textrm{ as defined in Section~\ref{subsec:tableTiling} where }  D_{\lambda_0}[V]\cap\Gamma= \{-1,0\}^2.
$$
The reader is referred to Proposition~\ref{prop:blockSubst_DilationDatum} for more details.
\end{example}

\begin{lemma}[\cite{BHP21-Symbolic}]
\label{lem:SuffLarge_General}
Let $\Dd=(G,d,(D_\lambda)_{\lambda>0}, \Gamma,V)$ be a dilation datum and $r_-,r_+>0$ be such that 
\begin{equation}
\label{eq:radii_r+_and_r-}
B(e,r_-) \subseteq V \subseteq \overline{V} \subseteq B(e,r_+).
\end{equation}
\begin{enumerate}[(a)]
\item If $\lambda_0>1+\frac{r_+}{r_-}$, then $\lambda_0$ is sufficiently large relative to $V$ with respect to $C_-=\frac{r_-}{\lambda_0}\left(\lambda_0-\big(1+\frac{r_+}{r_-}\big)\right)$, $s=0$ and $z=e$ for all $n\in\N$.
\item If there exists an $r> \frac{r_+\lambda_0}{\lambda_0-1}$, an $s\in \N_0$ and a $z\in\Gamma$ with $B(z,r)\subseteq V(s)$, then $\lambda_0$ is sufficiently large relative to $V$ with respect to $C_-=(r-r_+)-\frac{r}{\lambda_0}$, $s$ and $z$.
\end{enumerate}
\end{lemma}

Recall the notion $\Patstar:=\{P:M\to\Aa\,|\, M\subseteq\Gamma\}$. With this, one has 
 $\Aa^\Gamma\subseteq\Patstar$.

\begin{proposition}[Substitution map] \label{prop:substitutionmap}
\label{prop:SubstitutionMap}
Consider a substitution on $\Aa^\Gamma$ with substitution datum $\Ss:=(\Aa,\lambda_0, S_0)$ over a dilation datum $\Dd:=\big( G,d, (D_\lambda)_{\lambda>0}, \Gamma , V  \big)$. Then there exists a unique map $S:\Patstar\to\Patstar$ such that 
\begin{itemize}
\item $S^n(\gamma P) = D_{\lambda_0}^n(\gamma)S^n(P)$ holds for all $P \in\Patstar$ and $\gamma\in\Gamma$, (equivariance condition)
\item $\big( S^n(P) \big) \vert_{V(n,M)\cap \Gamma}= S^n \big( P\vert_M \big)$ for all $P \in \Patstar$, $n\in\N_0$ and $M\subseteq \Gamma$ nonempty. (restriction condition)
\end{itemize}
Moreover, the following holds.
\begin{enumerate}[(a)]
\item \label{enu:S_cont} The restriction $S:\Aa^\Gamma\to\Aa^\Gamma$ to $\Aa^\Gamma\subseteq\Patstar$ is continuous.
\item \label{enu:support_Sn_M} If $n\in\N_0$, $P\in\Patstar$ has finite support $M$, then $S^n(P)$ has support $V(n,M)\cap\Gamma$.
\item \label{enu:support_Sn_Equivariant} For all $n\in\N_0$, $\gamma\in\Gamma$ and $M\subseteq \Gamma$ finite, we have $V(n,\gamma M)= D^n(\gamma)V(n,M)$.
\end{enumerate}
\end{proposition}

\begin{proof}
This is proven in \cite[Proposition~2.7, Proposition~5.12, Lemma~5.16]{BHP21-Symbolic}.
\end{proof}

\begin{definition}
The unique map $S:\Patstar\to\Patstar$ in Proposition~\ref{prop:SubstitutionMap} associated with a substitution on $\Aa^\Gamma$ is called the {\em substitution map}.
\end{definition}

By convention, we identify a letter with a patch supported on $\{e\}$. With this convention at hand, one observes
$$
\Aa\subseteq \Patstar
\qquad\textrm{and}\qquad
S(a)=S_0(a).
$$

\begin{definition}[Legal patches]
\label{def:legalPatches}
Consider a substitution on $\Aa^\Gamma$ with substitution map $S$. A finite patch $P\in\Patstar$ is called {\em $S$-legal} if there is an $n\in\N$ and a letter $a\in\Aa$ satisfying $P\prec S^n(a)$. The set of all $S$-legal patches (with support $M\subseteq \Gamma$) is denoted by $W(S)$ respectively $W(S)_M$.
\end{definition}

Let $P,Q\in\Patstar$. Since the subpatch relation $\prec$ is transitive, we have for $n\in\N_0$,
$$ 
Q\in W(S) \textrm{ and } P\prec S^n(Q)
\qquad\Longrightarrow\qquad
P\in W(S),
$$
where $S^0(Q)=Q$. 
Legal patches are used to define an associated subshift 
$$
\Omega(S):=\{\omega\in\Aa^\Gamma\,|\, W(\omega)\subseteq W(S)\}
$$ 
with a substitution map $S$. 
We are mainly interested in primitive substitutions implying that the associated subshift $\Omega(S)$ is minimal. Here a subshift $\Omega\in\Jj$ is called {\em minimal}, if $\overline{\Orb(\omega)}=\Omega$ for all $\omega\in\Omega$. 

\begin{definition}[Primitive substitution]
\label{def:primitive}
A substitution on $\Aa^\Gamma$ with substitution map $S$ is called {\em primitive}, if there is an $L\in\N$ such that for all $a,b\in\Aa$, $a\prec S^L(b)$.
\end{definition}

In fact, if the substitution is primitive, then the subshift $\Omega(S)$ falls into a specific class of minimal subshifts introduced next.

\begin{definition}[Linearly repetitive subshifts]
\label{def:lin-rep}
A subshift $\Omega\in\Jj$ is called {\em linearly repetitive}, if there is a constant $C\geq 1$ such that for every $r\geq 1$, $R\geq Cr$ and $P\in W(\Omega)_{B(e,R)}$, we have 
\[	
Q\prec P \quad \text{for all} \quad Q\in W(\Omega)_{B(e,r)}.	
\] 
The smallest constant $\CLR=C\geq 1$ satisfying the previous condition is called the {\em linear repetitivity constant}.
\end{definition}

Let us summarize the previously discussed properties of the subshift associated with a substitution.

\begin{theorem}
\label{thm:Omega(S)}
Consider a substitution on $\Aa^\Gamma$ with substitution map $S$. Then $\Omega(S)\in\Jj$ is a subshift. If the substitution is primitive, then $\Omega(S)$ is minimal and linearly repetitive. In this case $W(S)=W(\omega)$ holds for all $\omega\in\Omega(S)$.
\end{theorem}

\begin{proof}
In the case of block substitutions (defined in Section~\ref{subsec:Block_Substitutions}) this is well-known see e.g. \cite[Chapter 5]{Queff10} and \cite[Chapter 4]{BaakeGrimm13}. In the general case, this is proven in \cite[Theorem~7.4]{BHP21-Symbolic}.
\end{proof}

\subsection{Substitution graphs and testing domains} 
\label{subsec:TestingDomains}

We introduce substitution graphs generalizing the graph that was introduced earlier for the table tiling substitution. Recall the basic graph theoretic notions introduced in Section~\ref{subsec:Block_Substitutions}.

\begin{definition}[Substitution graph]
\label{def:S-graphs}
Consider a substitution on $\Aa^\Gamma$ with substitution map $S$. For a finite nonempty set $K\subseteq \Gamma$ and an $n\in \N$, we define the {\em substitution graph} $G_S(K;n)$ by the directed graph with vertex set $\Vv=\Aa^K$ and the edge set $\Ee$ defined by
$$
(P,Q)\in \Ee \quad :\Longleftrightarrow\quad Q\prec S^{n}(P) \textrm{ and } P,Q \not\in W(S).
$$ 
\end{definition}

We once again emphasize that we only have edges between patches that are not $S$-legal (illegal). In particular, if all patches $\Aa^K$ are $S$-legal, then $G_S(K;n)$ has no edges.  More precisely, the $S$-legal patches are isolated vertices in $G_S(K;n)$.  
 Therefore, for any $\omega \in \mathcal{A}^\Gamma$, this graph keeps track of the illegal $K$-patches in $S^{n}(\omega)$ coming from the $K$-patches in $\omega$.
 For a correct choice of $K$ and $n$, this is the crucial part to determining whether $\overline{Orb\big( S^n(\omega_0) \big)}$ converges  to $\Omega(S)$ or does not. \\ 
In \cite{Ten24}, the reader can find some graphs plotted for $\Gamma=\Z$. 
The graph $\Gtable$ associated with the table tiling substitution (Section~\ref{subsec:Block_Substitutions}) is a substitution graph with $G_S(T,N_T)$ for $T=\{0,1\}^2$ and $N_T=1$. 

\medskip

We will be interested in the substitution graph $G_S(T;N_T)$ where $(T,N_T)$ satisfies two additional geometric properties with respect to the underlying dilation datum and the stretch factor $\lambda_0$. 
This leads to the concept of a testing tuple. 
Therefore, recall the notion of the supports $V(n,T)$ for $n\in\N$ and $T\subseteq\Gamma$ recursively defined in equation~\eqref{eq:recursive_V(n)}.

\begin{definition}
\label{def:TestingTuple}
A tuple $(T,N_T)$ is called a {\em testing tuple} (for $\Dd$ with associated stretch factor $\lambda_0>1$) if
\begin{enumerate}[(a)]
\item \label{enu:TestingDomain} $T\subseteq \Gamma$ is finite and for every $r>0$, there exists an $N_r(T)\geq 0$, such that for each $n\geq N_r(T)$ and $x\in \Gamma$, there is a $\gamma:=\gamma(n,x)\in \Gamma$ satisfying 
$$
x B(e,r)\subseteq D_{\lambda_0}^n \big( \gamma \big)V(n,T).
$$
\item \label{enu:N_T} $N_T\in\N$ is the smallest integer $m\in \mathbb{N}$ such that for every $x\in \Gamma$, there exists a $\gamma_x\in \Gamma$ satisfying
\[ 
x T  \subseteq D_{\lambda_0}^m(\gamma_x) V(m,T). 
\]
\end{enumerate}
If $(T,N_T)$ is a testing tuple, $T$ is also called {\em testing domain}.
\end{definition}	

\begin{remark}
Note that the testing tuple depends only on the dilation datum $\Dd$ and the associated stretch factor $\lambda_0>1$. Thus, it is independent of the alphabet $\Aa$ and the substitution rule.

Moreover, if $(T,N_T)$ is a testing tuple, then $(xT,N_T)$ is also a testing tuple for all $x\in\Gamma$. This follows immediately from the identity $V(n,xT) = D^n(x)V(n,T)$, see Proposition~\ref{prop:SubstitutionMap}~\eqref{enu:support_Sn_Equivariant}.
\end{remark}

\begin{example}
Recall the table tiling substitution introduced in Section~\ref{subsec:tableTiling} fitting in our setting (see Proposition~\ref{prop:blockSubst_DilationDatum}). Then a possible testing tuple of the table tiling is $(T,1)$ where $T=\{0,1\}^2$, see Proposition~\ref{prop:test-euclid}. Therefore the following theorem applies to the graph $\Gtable$ defined in Section~\ref{subsec:tableTiling}.
\end{example}

The notion of testing tuple is fundamental to check the convergence in our main Theorem~\ref{thm:charact_convergence}. In particular, it is a crucial ingredient in Lemma~\ref{lem:Patches_event_Legal} below.

In fact, the existence of a testing tuple is no restriction as it always exists, see Proposition~\ref{prop:TestingTuple_Existence}. We postpone the details for now but let us note the following. While a testing tuple always exists, it can be hard to find a ``minimal'' (in terms of cardinality) testing domain. We provide in Section~\ref{sec:Testing_Tuple_reduction} an algorithm \cite{TestHeisen24} to reduce the size of the testing domain and apply it in the case of the Heisenberg group.
Note that the smaller the set $T$, the smaller the vertex set of $G_S(T;N_T)$ is. 
Therefore it is computationally easier to check the condition \eqref{enu:finite-path} or \eqref{enu:path-length} in our first main Theorem~\ref{thm:charact_convergence} stated below.
We furthermore note that determining the growth of $r\mapsto N_r(T)$ is essential to prove the quantitative estimates in our second main Theorem~\ref{thm:upper-bound-estim}.

\subsection{First main result: Characterization of the convergence and its consequences} 
\label{subsec:DynSyst_SpectralConv}

We now have all at hand to formulate our first main result.

\begin{theorem} \label{thm:charact_convergence}
Consider a primitive substitution on $\Aa^\Gamma$ with substitution map $S$. For a testing tuple $(T,N_T)$ and $\omega_0\in\Aa^\Gamma$, the following assertions are equivalent.
\begin{enumerate}[(i)]
	\item \label{enu:SpectrConv} For all Schrödinger operators $H$ with finite range, we have
$$
\lim\limits_{n\to\infty}\sigma(H_{S^n(\omega_0)})=\sigma(H_{\omega}),\qquad \omega\in\Omega(S),
$$
	\item \label{enu:conv} $\lim\limits_{n\to\infty}\delta_H\big(\overline{\Orb(S^n(\omega_0))},\Omega(S)\big)=0$.
	\item \label{enu:finite-path} Each directed path in $G_S(T,N_T)$, starting in a vertex of $W(\omega_0)_T\subseteq \Aa^T$, does not contain a closed subpath.
	\item \label{enu:path-length} Each directed path in $G_S(T,N_T)$, starting in a vertex of $W(\omega_0)_T\subseteq \Aa^T$, is of length strictly less than $\big\vert \Aa^T \big\vert$.
\end{enumerate}
In particular, if $W(\omega_0)_T\subseteq W(S)$, then these equivalent conditions are satisfied.
\end{theorem}

We prove the theorem in Section~\ref{subsec:Pf_MainThm}.

\medskip

The equivalence of (\ref{enu:SpectrConv}) and (\ref{enu:conv}) is a consequence of \cite[Corollary~1]{BBdN18} using that every homogeneous substitution lattice is amenable. Our main contribution is the equivalence of (\ref{enu:conv}) to (\ref{enu:finite-path}) and (\ref{enu:path-length}). These conditions are of particular interest as they can be checked algorithmically as demonstrated on some examples in Sections~\ref{subsec:tableTiling} and~\ref{subsec:Heisenberg_Example}. Note further that one can consider a larger class of operators associated with the dynamical systems, see \cite{BBdN18}.

\medskip 

The assertions in Theorem~\ref{thm:charact_convergence} are particularly interesting in situations where one can numerically compute the spectrum explicitly for the approximations $H_{S^n(\omega_0)}$. This is for instance the case if $\Gamma=\Z^d$ and $S^n(\omega_0)$ is periodic (i.e. $\Orb(S^n(\omega_0))$ is finite) using the Floquet-Bloch theory.
We provide in Proposition~\ref{prop:Periodic_Heisenberg} an example of a primitive substitution on the Heisenberg group admitting periodic approximations. Together with the results on the bottom of the spectrum for periodic graphs given in \cite{Richter23}, our result might provide an approach to access the bottom of the spectrum in this case. This will be the subject of future investigations.

\medskip

This motivates the task of finding conditions for the existence of periodic approximations. In order to use Theorem~\ref{thm:charact_convergence}, we need the following.

\begin{proposition}
\label{prop:Periodic_Sn}
Consider a substitution on $\Aa^\Gamma$ with substitution map $S$. If $\omega_0\in\Aa^\Gamma$ is periodic, then $S^n(\omega_0)$ is periodic for all $n\in\N$. More precisely, if $M\subseteq \Gamma$ is finite satisfying $\Orb(\omega_0)=\{\eta\omega_0 \,|\, \eta\in M\}$, then
$$
\Orb(S^n(\omega_0))=\big\{ \eta S^n(\omega_0) \,|\, \eta \in (D^n[V]\cap\Gamma) \cdot D^n[M] \big\},
	\qquad n\in\N.
$$ 
\end{proposition}

\begin{proof}
Recall that $V$ is the fundamental domain of a uniform lattice and $D:=D_{\lambda_0}$ is the dilation with the stretch factor $\lambda_0>1$.

Let $H:=\Stab(\omega_0):=\{\gamma\in\Gamma \,|\, \gamma\omega_0=\omega_0\}$ be the stabilizer of $\omega_0$, which is a subgroup of $\Gamma$. Since $\Orb(\omega_0)$ is finite, there is a finite set $M\subseteq\Gamma$ such that $\Orb(\omega_0)=\{\eta\omega_0 \,|\, \eta\in M\}$.

We first show that $M\cdot H=\Gamma$. Clearly, $M\cdot H\subseteq\Gamma$ and so let $\gamma\in\Gamma$. Then there is an $\eta\in M$ such that $\gamma\omega_0=\eta\omega_0$ or equivalently $\eta^{-1}\gamma\omega_0=\omega_0$. By definition of the stabilizer, we conclude $\eta^{-1}\gamma\in H$, namely $\gamma\in\eta H\subseteq M\cdot H$.

Next, we prove that $S^n(\omega_0)$ is periodic.  Let $n\in\N$. Set $V_n:=D^n[V^{-1}]\cap\Gamma$, which satisfies $V_n\cdot D^n[\Gamma]=\Gamma$, since $V^{-1}\cdot \Gamma=(\Gamma \cdot V)^{-1}=G$ and $D$ is an automorphism. Hence,
$$
\Gamma=V_n\cdot D^n[M\cdot H] = V_n\cdot D^n[M]\cdot D^n[H],
$$ 
follows using again that $D$ is an automorphism. Define $M_n:=V_n\cdot D^n[M]\subseteq \Gamma$, which is a finite set as $V_n$ and $M$ are finite. Then for all $\gamma\in\Gamma$, there is an $\eta\in M_n$ and a $\gamma'\in H$ such that $\gamma=\eta D^n(\gamma')$. Thus, Proposition~\ref{prop:SubstitutionMap} and $\gamma'\in H$ lead to
$$
\gamma S^n(\omega_0) 
	= \eta D^n(\gamma')S^n(\omega_0)
	= \eta S^n(\gamma'\omega_0)
	= \eta S^n(\omega_0).
$$
Hence, $\Orb(S^n(\omega_0))=\{\eta\omega_0 \,|\, \eta\in M_n\}$ and so $S^n(\omega_0)$ is periodic since $M_n$ is finite.
\end{proof}

\begin{cor}
\label{cor:PeriodicallyApprox}
Consider a primitive substitution on $\Aa^\Gamma$ with substitution map $S$. 
If $\omega_0\in\Aa^\Gamma$ is periodic and satisfies one of the equivalent conditions in Theorem~\ref{thm:charact_convergence}, then $\Omega(S)$ is periodically approximable.
\end{cor}

\begin{proof}
By Theorem~\ref{thm:charact_convergence}, $\lim_{n\to\infty} \overline{\Orb(S^n(\omega_0))} =\Omega(S)$ holds. Since $\omega_0$ is periodic, we conclude that $S^n(\omega_0)$ is also periodic, see Proposition~\ref{prop:Periodic_Sn}. In particular, $\overline{\Orb(S^n(\omega_0))}=\Orb(S^n(\omega_0))$ holds since finite sets are closed.
\end{proof}

According to \cite[Proposition~2.11]{BBdN20}, every subshift associated with a primitive substitution on $\Aa^\Z$ is periodically approximable. Moreover, one can conclude from Theorem~\ref{thm:charact_convergence} that so-called self-correcting substitutions, see \cite[Definition 2.7]{GM13}, are periodically approximable. However, in general primitive substitutions do not give rise to periodically approximable subshifts in higher dimensions. 
This is  a direct consequence of Theorem~\ref{thm:charact_convergence} and a result from \cite{DuLeSh05,Oll08,Ballier23} see a summary in \cite{JeVa20}. In order to do so, we call a subshift $\Omega \in \Jj$ {\em strongly aperiodic} if $\Stab(\omega):=\{ \gamma\in\Gamma\,|\, \gamma\omega=\omega \}=\{ e\}$ for every $\omega\in \Omega$. Note that substitution subshifts are a typical class used to generate strongly aperiodic subshifts, see e.g. \cite{Sol98,Queff10,BaakeGrimm13,BHP21-Symbolic}.

\begin{cor}
\label{cor:PrimSubst_NotPerAppr}
There exists a strongly aperiodic, minimal subshift $\Omega(S)$ associated with a primitive substitution on $\Aa^{\Z^2}$ such that $\Omega(S)$ is not periodically approximable.
\end{cor}

\begin{proof}
We need to introduce some notations used in \cite{JeVa20}. 
A subset $\tau\subseteq \Aa^T$ with $T=\{0,1\}^2$ is called a {\em tileset}. A tileset $\tau$ is {\em intrinsically substitutive (with factor $2$)}, if there is a substitution on $\Aa^{\Z^2}$ with substitution rule $S_0:\Aa\to\Aa^T$ and substitution map $S$ satisfying
\begin{itemize}
\item $\Omega(S)=\Omega(\tau):=\{\omega\in\Aa^{\Z^2} \,|\, W(\omega)_T\subseteq \tau\}$,
\item for $P\in\Aa^M$ with $M\subseteq \Z^2$ finite, we have $S(P)\in W(S)$ if and only if $P\in W(S)$,
\item for all $\omega\in\Omega(\tau)$, there is a $\rho\in\Aa^{\Z^2}$ and a $\gamma\in\Z^2$ such that $\omega=\gamma S(\rho)$.
\end{itemize}
Note that $\Omega(\tau)$ is also called a subshift of finite type in the literature. Moreover, this substitution is actually a block substitution falling into our framework, see Section~\ref{subsec:Block_Substitutions}.

According to \cite{DuLeSh05,Oll08} (see also \cite[Proposition~8]{JeVa20}), there exists an intrinsically substitutive tileset $\tau$ with substitution map $S$ such that $\Omega(S)$ is strongly aperiodic. Due to \cite[Lemme~1.33]{Ballier23} this substitution is primitive. It is left to prove that $\Omega(S)$ is not periodically approximable. This can be done using (only) the first property of an intrinsically substitutive tilset together with its existence. In fact, by equation~\eqref{eq:HausdorffMetric_Dictionaries} and Theorem~\ref{thm:charact_convergence}, $\Omega(S)$ is periodically approximable if and only if there is a periodic $\omega_0\in\Aa^{\Z^2}$ such that $W(\omega_0)_T\subseteq W(S)_T\subseteq\tau$. By definition of $\Omega(\tau)=\Omega(S)$, we conclude $\omega_0\in\Omega(S)$. This contradicts that $\Omega(S)$ is strongly aperiodic. Summing up, $\Omega(S)$ is minimal, strongly aperiodic and not periodically approximable.
\end{proof}

Corollary~\ref{cor:PeriodicallyApprox} provides a sufficient condition for subshifts that are periodically approximable generalizing the results in \cite[Proposition~2.11]{BBdN20} and \cite[Theorem~6.2.3]{BeckusThesis}.
With the previous theorem at hand, we provide a  verifiable condition for primitive substitution systems in general to admit such periodic approximations. 
Moreover, Theorem~\ref{thm:charact_convergence} extends the result \cite[Theorem~6.2.3]{BeckusThesis} significantly.

Corollary~\ref{cor:PrimSubst_NotPerAppr} shows that there are substitutions on $\Aa^{\Z^2}$ that are not periodically approximable. Based on this, the question arises if such a behavior also occurs for non-abelian substitution subshifts. Specifically, are there primitive substitutions over a non-abelian group that are not periodically approximable?

\subsection{Second main result: Exponential rate of convergence for the spectra}
\label{subsec:MainResults}

The exponential convergence of the spectra for the table tiling substitution (Proposition~\ref{prop:Table_PeriodicAppr}), can be witnessed in large geometric generality. This is the core of the second main result of this paper, concerning the qualitative behavior of the convergence rate of $\sigma(H_{S^n(\omega_0)})$ to the spectrum $\sigma(H_\rho)$ for $\rho\in\Omega(S)$.

\medskip

These estimates are obtained by estimating the rate of the convergence of the subshifts $\overline{\Orb(S^n(\omega_0))}$ to $\Omega(S)$ in $(\Jj,\delta_H)$. This leads to the spectral estimates using the recent works \cite{BBC19,BecTak21}. The exponential decay obtained is given by the associated stretch factor $\lambda_0>1$.

\begin{theorem} 
\label{thm:upper-bound-estim}
Consider a primitive substitution on $\Aa^\Gamma$ with substitution map $S$.
Then there exist a $C>0$ and an $M_1 \geq 0$ such that if $\omega_0\in\Aa^\Gamma$ satisfies one of the equivalent conditions in Theorem~\ref{thm:charact_convergence}, then 
\[ 
\delta_H\big( \overline{\Orb(S^n(\omega_0))}, \Omega(S) \big)
	\leq  \frac{C}{\lambda_0^n},
	\qquad n\geq M_1.
\]
\end{theorem}

The constants $C$ and $M_1$ can be estimated more explicitly, see Proposition~\ref{prop:ConvergenceRate} below. It is also shown there that for specific initial configurations $\omega_0$ these constants can be significantly reduced, which is desirable for practical purposes, see e.g. Corollary~\ref{cor:RateConv_Zd}. We prove the theorem in Section~\ref{subsec:Rate_Converg}.

\begin{cor}
\label{cor:SpectralEstimates}
Consider a primitive substitution on $\Aa^\Gamma$ with substitution map $S$.
If $H$ is a strongly pattern equivariant Schrödinger operator with finite range, then there exist a $C'>0$ and an $N_0\in\N$ such that for all $\omega_0\in\Aa^\Gamma$ satisfying one of the equivalent conditions in Theorem~\ref{thm:charact_convergence}, we have
$$
d_H\big( \sigma(H_{S^n(\omega_0)}), \sigma(H_{\omega}) \big) \leq \frac{C'}{\lambda_0^n} ,\qquad \omega\in\Omega(S),\; n\geq N_0.
$$
\end{cor}

\begin{proof}
Since $S$ is primitive, $\Omega(S)$ is minimal (Theorem~\ref{thm:Omega(S)}) and so $\sigma(H_{\omega})=\sigma(H_{\rho})$ holds for all $\omega,\rho\in\Omega(S)$, see e.g.\@ \cite[Proposition~1.2.2]{Le99} in the abelian and \cite[Theorem~3.6.8]{BeckusThesis} in the non-abelian case. 
Now the estimate follows directly from Theorem~\ref{thm:upper-bound-estim} together with \cite[Theorem~1.3~(b)]{BecTak21} (see also \cite{BBC19} in the case $\Gamma=\Z^d$) using that all the coefficients are locally constant. Note that we can apply \cite{BecTak21} as $\Gamma$ has exact polynomial growth in our setting, see \cite[Proposition~3.36]{BHP21-Symbolic}.
\end{proof}

\begin{remark}
We note that the spectral estimate holds for a larger class of operators even with infinite range, see \cite{BBC19,BecTak21} for more details. Moreover, if the coefficients $t_\eta:\Aa^\Gamma\to\R, \eta\in B,$ of the Schrödinger operator $H$ are Lipschitz continuous (but $H$ is not necessarily strongly pattern equivariant), then 
$$
d_H\big( \sigma(H_{S^n(\omega_0)}), \sigma(H_{\omega}) \big) \leq \frac{C'}{\sqrt{\lambda_0^n}} ,\qquad \omega\in\Omega(S),\; n\geq N_0,
$$
follows for a suitable constant $C'>0$ and $N_0\in\N$, see \cite[Theorem~1.3~(a)]{BecTak21}.
\end{remark}

\section{Applications of the theory}
\label{sec:Block_and_Heisenberg}

In this section we provide explicit dilation data and substitution data to apply our main results. We give one class of abelian substitutions -- block substitutions, and another non-abelian example, namely a substitution over the discrete Heisenberg group.

\subsection{Block substitutions}
\label{subsec:Block_Substitutions}

We study a special class of substitutions in the abelian setting $\Gamma=\Z^d$ and prove that they are contained in the general class of substitutions introduced before. In particular, the table tiling substitution presented in Section~\ref{subsec:tableTiling} is an interesting example of block substitutions. The following two statements show that the theory developed in Section~\ref{sec:BeyondAbelian} is applicable for block substitutions and henceforth also for the table tiling substitution.

Let $\Aa$ be an alphabet. A {\em block substitution on $\Aa^{\Z^d}$} is defined by a vector $\mv=(m_1,\ldots,m_d)\in\N^d$ with $m_j>1$ for all $1\leq j\leq d$ and a {\em substitution rule}
$$
S_0:\Aa\to\Aa^{\Kmv} 
	\quad\textrm{where}\quad \Kmv:=\Z^d \cap \prod_{j=1}^d \big[-\frac{m_j}{2},\frac{m_j}{2}\big).
$$
It is standard to extend a substitution rule to a map $S:\PatstarZd\to\PatstarZd$ by acting letter wise, see e.g. \cite[Chapter 5.1]{Queff10} or \cite[Chapter 4]{BaakeGrimm13}. For instance, this is sketched in Figure~\ref{fig:TableTiling_Iterations} for the table tiling substitution. Thanks to the following proposition, block substitutions are a special case of the substitutions introduced in the previous section.

\begin{proposition}
\label{prop:blockSubst_DilationDatum} 
Consider a block substitution with alphabet $\Aa$, vector $\mv=(m_1,\ldots,m_d)\in\N^d$ with $m_j>1$ for all $1\leq j\leq d$ and substitution rule $S_0:\Aa\to\Aa^{\Kmv}$. Then the associated substitution map $S$ conincides with the substitution map arising from
\begin{itemize}
\item the dilation datum $\Dd:=\big(\R^d,d_{\mv},(D_\lambda)_{\lambda>0},\Z^d,V\big)$ with $V=\big[-\frac{1}{2},\frac{1}{2}\big)^d$, $\lambda_0=\min\{m_j\,|\, 1\leq j\leq d\}>1$, metric 
	$$
	d_{\mv}(x,y):= \max_{1\leq j \leq d} |x_j-y_j|^{\frac{1}{\alpha_j}}, 
		\qquad 
		x,y\in\R^d,
		\alpha_j:=\frac{\log(m_j)}{\log(\lambda_0)}\geq 1 \;\textrm{ for }\; 1\leq j\leq d,
	$$
	and the dilation family defined by
	$$
	D_\lambda(x):= \big(\lambda^{\alpha_1}x_1,\ldots, \lambda^{\alpha_d}x_d\big) 
		\qquad \textrm{for }\quad x=(x_1,\ldots,x_d)\in\R^d;
	$$
\item the substitution datum $\Ss=(\Aa,\lambda_0,S_0)$ with  $\Kmv = D_{\lambda_0}[V]\cap\Z^d$.
\end{itemize}
In particular, we have $D_{\lambda_0}(x)= (m_1x_1,\ldots,m_d x_d)$ for $x\in\R^d$ and $\lambda_0$ is sufficiently large relative to $V$ with respect to the constant 
$C_-:= 2\lambda_0-3$ and $s=4$.
\end{proposition}

\begin{proof}
Clearly, $d_{\mv}$ is a proper left-invariant metric inducing the Euclidean topology on $\R^d$ and
$$
d_{\mv}(D_\lambda(x),D_\lambda(y)) = \max_{1\leq j \leq d} \big| \lambda_0^{\alpha_j} x_j - \lambda_0^{\alpha_j} y_j \big|^{\frac{1}{\alpha_j}}
	= \lambda_0 d_{\mv}(x,y).
$$
It is straightforward to check that $D_\lambda:\R^d\to\R^d$ defines a group automorphism. Finally, the choice of $\alpha_j$ implies $ \log\big(\lambda_0^{\alpha_j}\big) = \log(m_j) $. We conclude that $\lambda_0^{\alpha_j}=m_j$ proving $D_{\lambda_0}(x)= (m_1x_1,\ldots,m_d x_d)$ for $x\in\R^d$ and $D_{\lambda_0}[V]\cap\Z^d=\Kmv$. 
Thus, $\Dd$ defines a dilation datum and $\Ss$ a substitution datum if we prove that  $\lambda_0$ is sufficiently large relative to $V$ with respect to $C_-:= 2\lambda_0-3$ and $s=4$.

In order to do so, we apply Lemma~\ref{lem:SuffLarge_General}~(b) for $s=4$. 
Therefore set $r_+=1$ satisfying $\overline{V}\subseteq B(e,r_+)$.
A short computation gives $V(4)=\prod_{j=1}^d [a_j,b_j)$ with $b_j-a_j=m_j^4$ for each $1\leq j\leq d$. Let $z_j\in\Z$ be chosen such that it attains the minimum $\min_{k\in\Z}\big|k-\frac{a_j + b_j}{2}\big|$, which is less or equal than $\frac{1}{2}$. Set $z=(z_1,\ldots,z_d)\in\Z^d$ and observe
$$
\prod_{j=1}^d (z_j-R_j,z_j+R_j) \subseteq \prod_{j=1}^d [a_j,b_j) = V(4) 
\qquad \textrm{where } 
R_j := \frac{m_j^4-1}{2}.
$$
Choose $r:=2\lambda_0>\frac{r_+\lambda_0}{\lambda_0-1}$. 
We show that
$$
B(z,r) 
= \big\{ \gamma\in\Z^d \,\big|\, d_{\mv}(\gamma,z)<r\big\}
= \big\{ \gamma\in\Z^d \,\big|\, |\gamma_j-z_j|<r^{\alpha_j}\big\}
\subseteq \prod_{j=1}^d (z_j-R_j,z_j+R_j)
\subseteq V(4).
$$
This is indeed the case since
$$
\log\big(2^{\alpha_j}\big) = \frac{\log(m_j)\log(2)}{\log(\lambda_0)} \leq \log(m_j)
\qquad\textrm{and}\qquad
\lambda_0^{\alpha_j}=m_j,
$$
lead to
$$
r^{\alpha_j} \leq m_j^2 \leq \frac{m_j^4-1}{2}=R_j.
$$
Thus, Lemma~\ref{lem:SuffLarge_General}~(b) implies that $\lambda_0$ is sufficiently large relative to $V$ with respect to $z\in\Z^d$, $s=4$ and $C_-:= r-r_+ - \frac{r}{\lambda_0}= 2\lambda_0-3$.
\end{proof}

Next, we prove that for the previously defined dilation datum $\Dd$ and stretch factor $\lambda_0$ for block substitutions, the tuple $(T,1)$ with $T=\{0,1\}^d$ is a testing tuple.

\begin{proposition}
\label{prop:test-euclid}
Let $\mv=(m_1,\ldots,m_d)\in\N^d$ with $m_j>1$ and consider the associated dilation datum $\Dd:=\big(\R^d,d_{\mv},(D_\lambda)_{\lambda>0},\Z^d,\big[-\frac{1}{2},\frac{1}{2}\big)^d\big)$ with associated stretch factor $\lambda_0:=\min\{m_j\,|\, 1\leq j\leq d\}>1$ found in Proposition~\ref{prop:blockSubst_DilationDatum}. Then $(T,1)$ with $T=\{ 0,1\}^d$ is a testing tuple and for each $r\geq 1$, $N_r(T):= \frac{\log(2\lambda_0r)}{\log(\lambda_0)}$ satisfies Definition~\ref{def:TestingTuple}~\eqref{enu:TestingDomain}.
\end{proposition}

\begin{proof}
We use the notation $D:=D_{\lambda_0}$ and $V:=\big[-\frac{1}{2},\frac{1}{2}\big)^d$. 

\underline{Step 1:} We first claim that for all $n\in\N$,
$$
\prod_{j=1}^d \big[-m_j^{n-1},m_j^n\big] \subseteq V(n,T).
$$
For the induction base, observe
$$
V(1,T) 
	= D\big[ T+V\big] 
	= D\left[ \left[ -\frac{1}{2},\frac{3}{2}\right)^d\right]
	= \prod_{j=1}^d \left[ -\frac{m_j}{2},\frac{3}{2}m_j\right)
	\supseteq \prod_{j=1}^d [-1,m_j],
$$
where $m_j\geq 2$ is used in the last step. For the induction step, the recursive definition of $V(n,T)$ in equation~\eqref{eq:recursive_V(n)} together with the induction hypothesis yields
\begin{align*}
V(n+1,T) 
	= &D\big[ \big(V(n,T)\cap\Z^d\big)+V\big] 
	\supseteq D\left[\left(\prod_{j=1}^d \big[-m_j^{n-1},m_j^n\big] \cap\Z^d\right)+V\right]\\
	= &D\left[\prod_{j=1}^d \left[-m_j^{n-1}-\frac{1}{2},m_j^n+\frac{1}{2}\right)\right]
	= \prod_{j=1}^d \left[-m_j^{n}-\frac{m_j}{2},m_j^{n+1}+\frac{m_j}{2}\right)\\
	\supseteq &\prod_{j=1}^d \big[-m_j^{n},m_j^{n+1}\big].
\end{align*}

\underline{Step 2:} Let $x\in\Z^d$, $r\geq 1$ and $n\geq N_r(T)=\frac{\log(2\lambda_0r)}{\log(\lambda_0)}$. For $\gamma=(\gamma_1,\ldots,\gamma_d)\in\Z^d$, Proposition~\ref{prop:blockSubst_DilationDatum} implies $D^n(\gamma) = (m_1^n\gamma_1,\ldots,m_d^n\gamma_d)$. Thus, there is a $\gamma\in\Z^d$ such that for any $y=(y_1, \ldots , y_d)\in \mathbb{Z}^d$, 
$$
	\max_{1\leq j\leq d}|x_j-y_j|\leq \frac{m_j^{n-1}}{2}
\qquad\Longrightarrow\qquad 
y\in D^n(\gamma) + \prod_{j=1}^d \big[-m_j^{n-1},m_j^{n}\big].
$$
Let $y\in\Z^d$ with $d_{\mv}(x,y)\leq r$, i.e., $y$ is an element of the ball $B_{\mv}(x,r)$ defined by the metric $d_{\mv}$. Since $n\geq N_r(T)$, we have $r\leq\frac{\lambda_0^{n-1}}{2}$. Hence, 
$$
|x_j-y_j|
	\leq r^{\alpha_j} 
	= \frac{(\lambda_0^{n-1})^{\alpha_j}}{2^{\alpha_j}}
	\leq \frac{(\lambda_0^{\alpha_j})^{n-1}}{2}
	= \frac{m_j^{n-1}}{2}
$$
follows for each $1\leq j\leq d$ by definition of the metric $d_{\mv}$. Thus, the inclusions 
$$
B_{\mv}(x,r)\subseteq D^n(\gamma) + \prod_{j=1}^d \big[-m_j^{n-1},m_j^{n}\big]
	\subseteq  D^n(\gamma) + V(n,T)
$$
follow by invoking Step 1. In particular, $T$ is a testing domain and for $r\geq 1$, $N_r(T)=\frac{\log(2\lambda_0r)}{\log(\lambda_0)}$ satisfies Definition~\ref{def:TestingTuple}~\eqref{enu:TestingDomain} by the previous considerations.

Using the induction base in Step 1 and $m_j>1$ for $1\leq j\leq d$, we conclude that $T\subseteq \prod_{j=1}^d [-1,m_j] \subseteq V(1,T)$. Thus, $N_T=1$ is the smallest integer satisfying Definition~\ref{def:TestingTuple}~\eqref{enu:N_T} proving that $(T,1)$ is a testing tuple.
\end{proof}

\subsection{An example in the Heisenberg group} 
\label{subsec:Heisenberg_Example}

Let $G:=\HeiR=\{(x,y,z)\,|\, x,y,z\in\R\}$ be the $3$-dimensional Heisenberg group with group multiplication defined by
\begin{equation*}
	(x,y,z)\cdot (a,b,c):= \Big( x+a,y+b, z+c + \frac{1}{2}(xb-ay)  \Big),\qquad (x,y,z),(a,b,c)\in \HeiR.
\end{equation*}
The Cygan-Kor{\'a}nyi norm on $\HeiR$ is defined by 
\begin{equation} \label{eq:Heisenberg_CyganKoranyiNorm}
\|\cdot\|_{CK}:\HeiR\to[0,\infty),\qquad  	 \|(x,y,z)\|_{CK} := \sqrt[4]{\big( x^2 + y^2 \big)^2 \, +\, z^2}.
\end{equation} 
This norm induces a left-invariant metric $d$ on $\HeiR$ via $d(g,h):= \|g^{-1}h\|_{CK}$ for $g,h\in G$. 
An underlying dilation family on $\HeiR$ is given by $(D_\lambda)_{\lambda>0}$ where
$$
D_\lambda:\HeiR\to\HeiR,\quad D_\lambda(x,y,z):=(\lambda x , \lambda y, \lambda^2 z ),
$$
The set $\Gamma:=\HeiZ$ of all vectors $(x,y,z)\in \HeiR$ with $x,y,z\in2\Z$ defines a uniform lattice in $\HeiR$ equipped with the invariant metric $d_\Gamma:=d_G|_{\Gamma\times\Gamma}$. Then $V:=[-1,1)^3$ defines a fundamental domain for $\Gamma$ in $\HeiR$. Clearly, we have $D_\lambda[\Gamma]\subseteq \Gamma$ if and only if $\lambda\in\N$. Altogether $\Dd_{H}=\big( \HeiR, d, (D_\lambda)_{\lambda>0}, \HeiZ, [-1,1)^3 \big)$ is a dilation datum.

\begin{figure}[htb] 
		\begin{center}
			\includegraphics[scale=1.1]{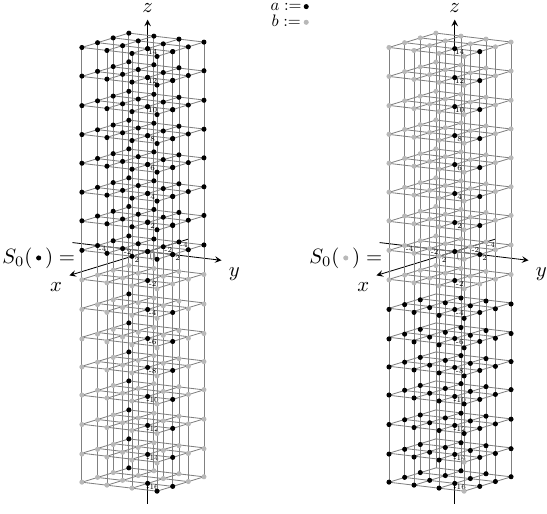}
		\end{center}
		\caption{A substitution rule where black bullet represents the letter $a$ and a gray bullet represents the letter $b$.}
		\label{fig:HeisenbergSubstitution}
	\end{figure}

\begin{proposition}
\label{prop:Periodic_Heisenberg}
Consider the dilation datum $\Dd_{H}=\big( \HeiR, d, (D_\lambda)_{\lambda>0}, \HeiZ, [-1,1)^3 \big)$ and the substitution datum $\Ss_{H}=(\Aa,\lambda_0,S_0)$ with stretch factor $\lambda_0=4$, alphabet $\Aa=\{a,b\}$ and substitution rule $S_0$ defined in Figure~\ref{fig:HeisenbergSubstitution}. Let $\omega_a,\omega_b\in\Aa^{\HeiZ}$ be defined by $\omega_a(\gamma)=a$ and $\omega_b(\gamma)=b$ for all $\gamma\in\HeiZ$. Then the following assertions hold.
\begin{enumerate}[(a)]
\item The subshift $\Omega(S)$ is strongly aperiodic (i.e.,\@ every point in $\Omega(S)$ has trivial $\Gamma$-stabilizer) and linearly repetitive. 
\item The set $T:=V(1)\cap\Gamma$ is a testing domain for $\Dd_H$ and $\lambda_0=4$.
\item The subshift $\Omega(S)$ is periodically approximable. In particular, there is a constant $C>0$ and $M_1\geq 0$ such that for $c\in\Aa$, we have the estimate
$$
\delta_H\big( \Orb\big(S^n(\omega_c)\big), \Omega(S)\big) 
	\leq C \frac{1}{4^n},
	\qquad n\geq M_1.
$$
\end{enumerate}
\end{proposition}

The proof of part~(c) makes use of a computer algorithm which for a given testing domain finds a proper subset that is also a testing domain, cf.\@ Proposition~\ref{prop:Heisen-test-dom}. In this sense, the proof of part~(c) is computer-assisted.  

\begin{proof}
A short computation leads to $$
B(e,r_-)\subseteq V \subseteq \overline{V} \subseteq B(e,r_+),\qquad 
	\textrm{with }\quad r_-=1 \quad\textrm{ and }\quad r_+=\frac{3}{2}.
$$
Thus, Lemma~\ref{lem:SuffLarge_General}~(a) implies that if $\lambda_0\in\N$ satisfies $\lambda_0\geq 1+\frac{r_+}{r_-}=\frac{5}{2}$, then $\lambda_0$ is sufficiently large relative to $V$ with respect to $s=0$, $z=e$ and the constant
$$
C_-=\frac{r_-}{\lambda_0} \left( \lambda_0-\big(1+\frac{r_+}{r_-}\big) \right) 
	= \frac{2\lambda_0-5}{2\lambda_0}
	=\frac{3}{8}.
$$
Observe that $r_+ =\lambda_0 C_-$. Now choose $\delta > 0$ small enough such that $V B(e,\delta) \subseteq B(e,r_{+}) = B(e,C_{-} \lambda_0)$. 
Then Proposition~\ref{prop:TestingTuple_Existence}~(b) applied with this $\delta > 0$, $s_1=0$ and $s_2=1$ implies that
$$
T':=V(1)\cap\Gamma = D_{\lambda_0}[V]\cap\Gamma = \{-4,-2,0,2\}^2\times\{-16,-14,-12,\ldots,14\}
$$
is a testing domain proving (b).

(a) The previously defined substitution datum is a good substitution datum in the sense of \cite[Definition~6.4]{BHP21-Symbolic}. Keeping the notation of the last work, 
define
\begin{align*}
\Xi_a := &\big\{ (0,0,z) \,\big|\, z\in[-16,16)\cap 2\Z \big\}\bigcup \big\{ (0,2,z) \,\big|\, z\in [-16,16)\cap 2\Z \big\} \setminus  \{ (0,2,-2) \},\\
\Xi_o := &\big\{(x,y,-2) \,\big|\, x,y\in[-4,4)\cap 2\Z \big\}\setminus \{ (0,0,-2) \}, \\
&(\gamma_2,x_b):=(2,2,-14), \qquad (\gamma_2,x_2):=(2,2,-16).
\end{align*}
With these choices, the conditions in \cite[Definition~6.4]{BHP21-Symbolic} are satisfied. Indeed, for $c\in\Aa$, we have $S_0(c)(\gamma_2,x_2)=c$, $S_0(c)(\gamma)= b$ for all $\gamma\in \Xi_o\cup\{(\gamma_2,x_b)\}$ and  $S_0(c)(\gamma)= a$ for all $\gamma\in \Xi_a$.  
Hence, \cite[Proposition~6.6, Theorem~1.4, Theorem~1.6]{BHP21-Symbolic} asserts that $\Omega(S)$ is strongly aperiodic and linearly repetitive. 

(c) We first show that $\lim_{n\to\infty} \Orb\big(S^n(\omega_c))=\Omega(S)$ for $c\in\Aa$. Towards this, a smaller testing domain is helpful. Proposition~\ref{prop:Heisen-test-dom} asserts that
$$
T':=\{-2,0\}^2\times\{-6,-4,-2,0,2,4,6\}
$$
is a testing domain for $\Dd_H$ and $\lambda_0=4$. Note that here, we used a computer to reduce the testing domain. 
Then $W(\omega_a)_{T'}=\{P\}$ and $W(\omega_b)_{T'}=\{Q\}$ follows where the patches $P,Q\in\Aa^{T'}$ are defined by $P(\gamma):=a$ and $Q(\gamma):=b$ for all $\gamma\in {T'}$.
To prove convergence of the dynamical systems, it suffices by Theorem~\ref{thm:charact_convergence} to show $W(\omega_a)_{T'}\subseteq W(S)$ and $W(\omega_b)_{T'}\subseteq W(S)$, namely that $P,Q\in\Aa^{T'}$ are $S$-legal. Direct computations imply
$$
(0,0,6){T'} = \{-2,0\}^2\times \{0,2,4,\ldots,10,12\} =:K_0
$$
and
$$
(-2,-2,6){T'}\subseteq \{-4,-2\}^2\times \{-2,0,\ldots,12,14\}=:K_1.
$$
From the substitution rule $S_0$ (Figure~\ref{fig:HeisenbergSubstitution}), we conclude
$$
P\prec S_0(a)|_{K_0} \qquad\textrm{and}\qquad Q\prec S_0(b)|_{K_1},
$$
proving that $P,Q\in W(S)$, namely these patches are $S$-legal. Now statement (c) follows from Theorem~\ref{thm:upper-bound-estim}.
\end{proof}

\section{Proof of the first main result} 
\label{sec:Cond-sect}

In this section, we prove Theorem~\ref{thm:charact_convergence}. Throughout this section, we fix a substitution on $\Aa^\Gamma$ with substitution map $S$. More precisely, let $\Dd=\big( G,d, (D_\lambda)_{\lambda>0}, \Gamma, V \big)$ be the associated dilation datum and $\Ss=(\Aa,\lambda_0,S_0)$ be the associated substitution datum. We use the notation $D:=D_{\lambda_0}:\Gamma\to\Gamma$.

\subsection{Sufficient condition}
We start by generalizing an unpublished result on block substitutions on $\Z^d$ by one of the authors \cite[Theorem~6.2.3]{BeckusThesis} providing a sufficient condition for the convergence of the subshifts. The special case $d=1$ was treated in \cite[Corollary~5.5]{BBdN20}.

\begin{lemma} 
\label{lem:Patches_event_Legal}
Let $(T,N_T)$ be a testing tuple of a substitution on $\Aa^\Gamma$ with substitution map $S$. For $r>0$, let $N_r(T)\geq 0$ be such that it satisfies Definition~\ref{def:TestingTuple}~\eqref{enu:TestingDomain}. Then for each $r > 0$ and for all $\omega_0 \in \Aa^{\Gamma}$ with $W(\omega_0)_T\subseteq W(S) $, we have
\[  
W\big(S^n(\omega_0) \big)_{B(e,r)} \subseteq W(S) 
	\quad \text{for all} \quad n\geq N_r(T). 
\]
\end{lemma}

\begin{proof}
Let $r >0$ and $P\in W(S^n(\omega_0))_{B(e,r)}$. Hence, there is an $x\in \Gamma$ such that $P= S^n(\omega_0) \vert_{xB(e,r) } $.
By the choice of $N_r(T)$, for each $n\geq N_r(T)$, there is a $\gamma:=\gamma(n,x)\in \Gamma$ satisfying 
$$
x B(e,r)\subseteq D^n \big( \gamma \big)V(n,T)=V\big( n, \gamma T \big)
$$
using Proposition~\ref{prop:SubstitutionMap}~\eqref{enu:support_Sn_Equivariant}. Then Proposition~\ref{prop:SubstitutionMap} implies
\begin{equation*}
P 
	= \big( S^n(\omega_0) \big)\vert_{x B(e,r)}
	\prec \big( S^n(\omega_0) \big)\vert_{V( n, \gamma T )\cap \Gamma } 
	=S^n \Big(  \omega_0 \vert_{\gamma  T} \Big). 
\end{equation*}
Since $W(\omega_0)_T\subseteq W(S)$, the patch $\omega_0 \vert_{\gamma  T}$ is $S$-legal and so $P \in W(S)$.
\end{proof}

When our starting configuration $\omega_0$ and our substitution satisfy rather mild conditions, we are guaranteed that our approximations $S^n(\omega_0)$ eventually contain all legal patches for any fixed sized support. Recall that every letter $a\in\Aa$ is viewed as patch with support $\{e\}$.

\begin{lemma} 
\label{lem:LegalPatches_Appear}
Consider a substitution on $\Aa^\Gamma$ with associated substitution map $S$.
If for $\omega_0\in\Aa^\Gamma$, there is an $n_0\in\N$ such that $\Aa\subseteq W(S^n(\omega_0))$ for all $n\geq n_0$, then for each $r>0$, there exists an $M_r\geq n_0$ satisfying 
\[ 
W\big(S^n(\omega_0)\big) \supseteq W(S)_{B(e,r)}   
\quad \text{for all} \quad n\geq M_r. 
\]
In particular, if the substitution rule $S_0$ is primitive, then the constant $M_r$ can be chosen independently of $\omega_0$.
\end{lemma}

\begin{proof}
Let $r>0$. By definition of $S$-legal patches, for each $Q\in W(S)_{B(e,r)}$, there exists an $m_Q\in \mathbb{N}$ and a letter $a_Q\in \Aa$ such that $Q\prec S^{m_Q}(a_Q)$. Since $W(S)_{B(e,r)}$ is finite, the following maximum exists
\[ 
m_r:= \max_{Q\in W(S)_{B(e,r)}} m_Q. 
\]	
Set $M_r:=n_0+m_r$.  
Let $n\geq M_r$. For $Q\in W(S)_{B(e,r)}$, let $a_Q\in \Aa$ be chosen as before. Since $\Aa\subseteq W(S^n(\omega_0))$ for $n\geq n_0$, there is a letter $b\in\Aa$ such that $b\prec \omega_0$ and $a_Q\prec S^{n-n_Q}(b)$ holds since $n-m_r \geq n_0$. Hence,
\[
Q
	\prec S^{n_{Q}}(a_Q)
	\prec S^{n_Q} \big( S^{n-n_Q}(b) \big) 
	= S^n(b)
	\prec S^n (\omega_0)
\]
follows implying $Q\in W\big(S^n(\omega_0)\big)$ since $\prec$ is transitive. Since $Q\in W(S)_{B(e,r)}$ was arbitrary, we conclude $W(S)_{B(e,r)}\subseteq W(S^n(\omega_0))$ for $n\geq M_r$.

The case when the substitution rule $S_0$ is primitive follows then directly from the definition by setting $n_0=L$, for the $L\in \N$ in Definition~\ref{def:primitive}.
\end{proof}

With this at hand we can now generalize \cite[Theorem~6.2.3]{BeckusThesis} to our setting, and establish a sufficient condition for convergence of the iterative approximations.

\begin{proposition} 
\label{prop:suff-cond_Conv}
Let $(T,N_T)$ be a testing tuple for a primitive substitution on $\Aa^\Gamma$ with substitution map $S$. If for $\omega_0\in\Aa^\Gamma$, $W(\omega_0)_T\subseteq W(S)$, then 
\[ 
\lim_{n\to\infty} \overline{ \Orb\big( S^n(\omega_0) \big) } = \Omega(S) \text{ in } \Jj.
\] 
\end{proposition}	

\begin{proof}
Due to equation~\eqref{eq:HausdorffMetric_Dictionaries}, it suffices to prove that for every $r>0$, there exists an $N_r(T)$ such that $W(S^n(\omega_0))_{B(e,r)}=W(S)_{B(e,r)}$ if $n\geq N_r(T)$. The existence of $N_r(T)$ follows from Lemma~\ref{lem:Patches_event_Legal} and Lemma~\ref{lem:LegalPatches_Appear}.
\end{proof}

As a consequence of the previous proposition, we conclude that if any subsequence of iterative approximations converges then the whole sequence converges.

\begin{cor} 
\label{cor:subseq-cond}
Let $\Ss$ be a primitive substitution datum over a dilation datum $\Dd$. Let $\omega_0\in \Aa^\Gamma$ and denote $\Omega_n:= \overline{ \Orb\big( S^{n}(\omega_0) \big) }$. Then the following assertions are equivalent.
\begin{enumerate}[(i)]
	\item The iterative approximation sequence, $\Omega_n$, converges to $\Omega(S)$.
	\item There exists a subsequence $\big( \Omega_{n_k} \big)_{k=1}^\infty$ converging to $\Omega(S)$.
\end{enumerate} 
\end{cor}

\begin{proof}
(i)$\Rightarrow$(ii): This is obvious.

\medskip

(ii)$\Rightarrow$(i): Using Proposition~\ref{prop:suff-cond_Conv}, it suffices to show that there is an $n_0\in \mathbb{N}$ such that $W\big( \Omega_{n_0} \big)_T\subseteq W(S)$, where $T$ is a testing domain. Since $T$ is finite, there exists an $r_0>0$ such that $T\subseteq B(e,r_0)$. Since $\Omega_{n_k}\to \Omega(S)$, there is a $k_0\in \N$ by equation~\eqref{eq:HausdorffMetric_Dictionaries}, such that $W(\Omega_{n_k})_{B(e,r_0)}=W\big( \Omega(S) \big)_{B(e,r_0)}$ for all $k\geq k_0$. Thus, setting $n_0=n_{k_0}$ finishes the proof.
\end{proof}

\begin{remark}
\label{rem:RelaxAss_CharConv1}
One can replace the assumption of primitivity in Proposition~\ref{prop:suff-cond_Conv} and Corollary~\ref{cor:subseq-cond} by the weaker assumption in Lemma~\ref{lem:LegalPatches_Appear}. More precisely, if for $\omega_0\in\Aa^\Gamma$, there exists an $n_0\in\N$ such that $\Aa\subseteq W(S^n(\omega_0))$ for all $n\geq n_0$, then the statements stay valid. 
With this at hand, one can prove that the conditions (ii), (iii) and (iv) in Theorem~\ref{thm:charact_convergence} are still equivalent under this weaker assumption.
\end{remark}

\subsection{Proof of Theorem~\ref{thm:charact_convergence}}
\label{subsec:Pf_MainThm}

Before proving Theorem~\ref{thm:charact_convergence}, we show that cycles in the substitution graph $G_S(T,N_T)$ for a testing tuple $(T,N_T)$ imply that illegal patches reappear infinitely often when applying the substitution map.

\begin{lemma} 
\label{lem:prev-step}
Let $(T,N_T)$ be a testing tuple of a substitution on $\Aa^\Gamma$ with substitution map $S$ and $\omega_0\in\Aa^\Gamma$. If $P\in W\big( S^{N_T}(\omega_0) \big)_T \setminus W(S)$, then there exists a $Q\in W(\omega_0)_T\setminus W(S)$ such that $P\prec S^{N_T}(Q)$. 
\end{lemma}

\begin{proof}
We know that $\text{supp}(P)= x T$ for some $x\in \Gamma$. Since $(T,N_T)$ is a testing tuple, there exists a $\gamma =\gamma(x)\in \Gamma$ such that $x T\subseteq D^{N_T}(\gamma)  V(N_T,T) =V(N_T,\gamma T)$. Define $Q\in W(\omega_0)_T$ by $Q:= \omega_0\vert_{ \gamma T } $. By Proposition~\ref{prop:SubstitutionMap}, we conclude that
\begin{align*}
P =\Big( S^{N_T}(\omega_0) \Big)|_{ x T } 
	\prec \Big( S^{N_T}(\omega_0) \Big)|_{ V(N_T,\gamma T)\cap \Gamma }
	= S^{N_T}\Big( \omega_0 |_{\gamma  T}  \Big)
	= S^{N_T}(Q).
\end{align*}
Since $P\prec S^{N_T}(Q)$, while $P\notin W(S)$, $Q\notin W(S)$ follows.
\end{proof}

The last lemma allows to find suitable paths in the graph $G_S(T,N_T)$ of length $n\in\N$ if there is a patch in $W(S^{n\cdot N_T}(\omega_0))_T\setminus W(S)$. Recall that the vertex set of $G_S(T,N_T)$ is $\Aa^T$.

\begin{lemma}\label{lem:paths-ileg} 
Let $(T,N_T)$ be a testing tuple of a substitution on $\Aa^\Gamma$ with substitution map $S$ and substitution graph $G_S(T,N_T)$. If $P\in W\big(S^{n\cdot N_T}(\omega_0)\big)_T\setminus W(S)$ holds for $\omega_0\in\Aa^\Gamma$ and $n\in\N$, then there exists a path $(P_0,\ldots,P_n)$ in $G_S(T,N_T)$ of length $n$ such that $P_n=P$ and $P_0\in W(\omega_0)_T$.
\end{lemma}
	
\begin{proof}
The path is defined recursively. Set $P_n:=P$. For $j\geq 1$, define $\omega_{j}:=S^{j\cdot N_T}(\omega_0)$ and observe $S^{N_T}(\omega_{j-1}) = \omega_{j}$. 
Suppose $P_j\in W(\omega_j)\setminus W(S)$ for some $j\geq 1$. By Lemma~\ref{lem:prev-step} applied to $\omega_{j-1}$, there is a patch $P_{j-1}\in W(\omega_{j-1})\setminus W(S)$ such that $P_j\prec S^{N_T}(P_{j-1})$. Since $P_{j-1},P_j\not\in W(S)$, there is an edge from $P_{j-1}$ to $P_j$, see Definition~\ref{def:S-graphs}. Thus, we have recursively constructed a path $(P_0,\ldots,P_n)$ in $G_S(T,N_T)$ with $P_n=P$ and $P_0\in W(\omega_0)_T$.
\end{proof}
	
\begin{proof}[Proof of Theorem~\ref{thm:charact_convergence}] 
Since $S$ is primitive, $\Omega(S)$ is minimal (Theorem~\ref{thm:Omega(S)}) and so $\sigma(H_{\omega})=\sigma(H_{\rho})$ holds for all $\omega,\rho\in\Omega(S)$ and for every Schrödinger operator $H$ with finite range, see e.g. \cite[Proposition~1.2.2]{Le99} in the abelian and \cite[Theorem~3.6.8]{BeckusThesis} in the non-abelian case.  
Thus, the equivalence of (\ref{enu:SpectrConv}) and (\ref{enu:conv}) follows from \cite[Theorem~2]{BBdN18}. Note that we can apply the result in the reference as $\Gamma$ is amenable in our setting, see \cite[Section~3.4]{BHP21-Symbolic}.

We continue to prove the equivalences of (\ref{enu:conv}), (\ref{enu:finite-path}) and (\ref{enu:path-length}) via contraposition.

(\ref{enu:conv})$\Rightarrow$(\ref{enu:finite-path}): Assume towards contraposition that (\ref{enu:finite-path}) is not true. Then there is a path $(P_0,\ldots,P_\ell)$ in $G_S(T,N_T)$ with $P_0\in W(\omega_0)_T$ and $P_i=P_j$ for some $0\leq i<j\leq \ell$. By Definition~\ref{def:S-graphs} of the substitution graph $G_S(T,N_T)$, we conclude
$$
P_{k+1}\prec S^{N_T}(P_k) 
	\qquad\textrm{and}\qquad
P_k\not\in W(S).
$$
Thus, $P_i = P_j \prec S^{(j-i)\cdot N_T}(P_i)$ follows. Hence, we inductively conclude $P_i\prec S^{m\cdot (j-i)\cdot N_T}(P_i)$ for all $m\in\N$. On the other hand, $P_i\prec S^{i\cdot N_T}(P_0)\prec S^{i\cdot N_T}(\omega_0)$ holds as $(P_0,\ldots,P_i)$ is a path in $G_S(T,N_T)$. Set $n_m:=\big(i+m\cdot(j-i)\big)\cdot N_T$ for $m\in\N$ where $j-i \geq 1$ by construction. Then $n_m\to \infty$ if $m\to\infty$ and $P_i\prec S^{n_m}(\omega_0)$, namely $P_i\in W(S^{n_m}(\omega_0))_T$. Since $P_i\not\in W(S)$, the subsequence $\big(\overline{\Orb(S^{n_m}(\omega_0))}\big)_{m\in\N}$ does not converge to $\Omega(S)$ by equation~\eqref{eq:HausdorffMetric_Dictionaries}. Thus, $\overline{\Orb(S^n(\omega_0))}$ does not converge to $\Omega(S)$ by Corollary~\ref{cor:subseq-cond}, namely (\ref{enu:conv}) is not true.

\medskip

(\ref{enu:finite-path})$\Rightarrow$(\ref{enu:path-length}): Assume towards contraposition that (\ref{enu:path-length}) is not true. Hence there exists a path  $(P_0,...,P_n)$ in $G_S(T)$ such that $P_0\in W(\omega_0)_T$ and $n\geq \vert \Aa^T\vert$ where $\vert \Aa^T\vert$ is the total number of vertices in $G_S(T,N_T)$. Thus, there exist distinct  $i,j\in \{0,1,...,n\}$ such that $P_i=P_j$. We have therefore found a path in $G_S(T,N_T)$ starting at some $P_0\in W(\omega_0)_T$, which contains a closed subpath, namely (\ref{enu:finite-path}) is not true.
 
\medskip

(\ref{enu:path-length})$\Rightarrow$(\ref{enu:conv}): Assume towards contraposition that (\ref{enu:conv}) is not true. We know  that $W(S^n(\omega_0))_T\nsubseteq W(S)$ for all $n\in \mathbb{N}$, otherwise  it would follow that $\lim_{n\to\infty} \overline{\Orb(S^n(\omega_0))}= \Omega(S)$ from Proposition~\ref{prop:suff-cond_Conv}. Choosing $n=\vert \Aa^T\vert^{N_T}$, there must be a $P_n\in W(S^{n}(\omega_0))_T\setminus W(S)$. Thus, Lemma~\ref{lem:paths-ileg} implies that there is a path $(P_0,...,P_n)$ in $G_S(T,N_T)$ of length $\vert \Aa^T \vert$. Therefore, (\ref{enu:path-length}) does not hold. 		
\end{proof}

\section{Proof of the second main result} 
\label{sec:Rate_Converg}

We prove Theorem~\ref{thm:upper-bound-estim} by proving a quantitative version of Lemma~\ref{lem:Patches_event_Legal} and Lemma~\ref{lem:LegalPatches_Appear}. We start by further examining testing domains. We prove their existence and how the associated map $r\mapsto N_r(T)$ grows. The arguments used in this section can be compared with the explanations for block substitutions, see Proposition~\ref{prop:test-euclid}. These estimates are used to show the explicit quantitative estimates in Theorem~\ref{thm:upper-bound-estim}. To this end, a thorough study of the supports $V(n,T)$ is necessary.

\subsection{More on testing domains} 
\label{subsec:TestingDomain_Existence}

The following statement was mainly proven in \cite{BHP21-Symbolic} or follows by similar arguments.

\begin{lemma}
\label{lem:SupportGrowth}
Let $\Dd=(G,d,(D_\lambda)_{\lambda>0},\Gamma,V)$ be a dilation datum and $\lambda_0>1$ be such that $D_{\lambda_0}[\Gamma]\subseteq\Gamma$. Then the following statements hold for all $n\in\N_0$, $\gamma\in\Gamma$ and $M,M'\subseteq \Gamma$ finite and nonempty.
\begin{enumerate}[(a)]
\item \label{enu:V(n,V(m,M))} For $m\in\N_0$, we have $V\big( n,V(m,M)\cap \Gamma \big)=V(n+m,M)$.
\item \label{enu:M_subset_M'} If $M\subseteq M'$, then $V(n,M)\subseteq V(n,M')$.
\item \label{enu:M_cap_M'} If $M\cap M'= \emptyset$, then $V(n,M\sqcup M')=V(n,M)\sqcup V(n,M')$.
\end{enumerate}
\end{lemma}

\begin{proof}
Statements \eqref{enu:V(n,V(m,M))} and \eqref{enu:M_subset_M'} are proven in \cite[Lemma~5.16]{BHP21-Symbolic}. Statement \eqref{enu:M_cap_M'} follows by induction. The claim is true for $n=0$, since $V$ is a fundamental domain. Assume that the claim holds for $n$. Since $D$ is an automorphism and $G=\sqcup_{\gamma \in \Gamma} D(\gamma)D[V] $, we conclude for the induction step
\begin{align*}
V(n+1,M\sqcup M')&= D \Big[ \big( V(n,M\sqcup M') \cap \Gamma \big)\cdot V  \Big] \\
	&=D\Big[ \big( V(n,M) \cap \Gamma \big) \sqcup \big( V(n, M') \cap \Gamma \big) \Big]\cdot D[V]\\
	&= \Big( D\big[ V(n,M)\cap \Gamma \big] \cdot D[V] \Big)\sqcup \Big( D\big[ V(n,M')\cap \Gamma \big] \cdot D[V] \Big)\\
	&=V(n+1,M)\sqcup V(n+1,M').
\end{align*}
\end{proof}

\begin{remark}
Note that Lemma~\ref{lem:SupportGrowth}~\eqref{enu:M_subset_M'} implies that if $T$ is a testing domain and a finite $T'\subseteq \Gamma$ satisfies $T\subseteq T'$, then $T'$ is also a testing domain.
\end{remark}

Now we can prove an analog of Proposition~\ref{prop:test-euclid} for general substitutions.
	
\begin{proposition}[Existence of a testing tuple]
\label{prop:TestingTuple_Existence}
Let $\Dd=(G,d,(D_\lambda)_{\lambda>0},\Gamma,V)$ be a dilation datum and $\lambda_0>1$ be such that $D_{\lambda_0}[\Gamma]\subseteq\Gamma$ and $\lambda_0$ is sufficiently large relative to $V$ with respect to $C_->0$, $s_1\in\N_0$ and $z\in\Gamma$, i.e., $D^n\big[B(z,C_-)\big]\subseteq V(s_1+n)$ for all $n\in\N$. Fix $\delta > 0$. Then the following hold. 
\begin{enumerate}[(a)]
\item There exist an $s_2\in\N$ such that $V \cdot B(e,\delta) \subseteq B\big( e , C_- \lambda_0^{s_2}\big)$.
\item If $\delta>0$ and $s_2\in\N$ satisfy the inclusion in (a), then $T:=V(s_1+s_2)\cap\Gamma$ is a testing domain and for each $r\geq 1$, $N_r(T):= \frac{\log(r C_T )}{\log(\lambda_0)}$ with $C_T:=\frac{1}{\delta}$ satisfies Definition~\ref{def:TestingTuple}~\eqref{enu:TestingDomain}. 
\end{enumerate}
In particular, there exsists a testing tuple $(T,N_T)$ for $\Dd$ and $\lambda_0$.
\end{proposition}

\begin{proof}
Recall that if we prove that there is a testing domain $T$ for the given dilation datum and $\lambda_0$, then there exists an integer $N_T$ satisfying Definition~\ref{def:TestingTuple}~\eqref{enu:N_T}. Thus, if we show that $T$ is a testing domain, it follows that $(T,N_T)$ is a testing tuple for $\Dd$ and $\lambda_0$. 

Since $V \cdot B(e,\delta)$ is relatively compact and $\lambda_0>1$, there exists an $s_2\in\N$ such that $V \cdot B(e,\delta) \subseteq B\big( e , C_-\lambda_0^{s_2}\big)$, which is possible as $\lambda_0>1$. We now show that for any such $s_2\in\N$, $T:=V(s_1+s_2)\cap\Gamma$ defines a testing domain. 

Let $x\in\Gamma$, $r\geq 1$ and set $N_r(T):= \frac{\log(rC_T )}{\log(\lambda_0)}$ with $C_T:=\frac{1}{\delta}$. 
Let $n\geq N_r(T)$.  
Since $\lambda_0$ is sufficiently large relative to $V$ with respect to $C_- > 0$, $s_1\in\N_0$ and $z\in\Gamma$, we have $D^n\big[B(z,C_-)\big]\subseteq V(s_1+n)$. 
Set $s_0:=s_1+s_2$ and $z_0:=D^{s_2}(z)\in\Gamma$. Then $D^n\big[B(z_0,C_-\lambda_0^{s_2})\big]\subseteq V(s_0+n)$ follows from the previous considerations. Since $z_0\in \Gamma$ and $V$ is a fundamental domain of $\Gamma$, we conclude that $\Gamma \cdot (z_0 V)=G$. Thus, there exists a $\gamma \in \Gamma$ such that $D^{-n}(x)\in \gamma z_0  V$. Then the choice of $s_2$ implies
$$
D^{-n}(x) B(e,\delta) 
	\subseteq \gamma z_0 V \cdot B(e,\delta) 
	\subseteq  \gamma  z_0 B\big( e , C_-\lambda_0^{s_2} \big).
$$
Since $n\geq N_r(T)$, we conclude $\delta \lambda_0^n \geq r$. Using that $D$ is a group automorphism, we derive that
\begin{align*}
x B(e,r) 
	&\subseteq x  B(e,\delta \lambda_0^n) 
	= D^n \Big[ D^{-n}(x) B(e,\delta) \Big]
	\subseteq D^n \Big[\gamma  z_0 B\big( e , C_- \lambda_0^{s_2} \big) \Big] \\
	 	&= D^n(\gamma) D^n\big[B(z_0,C_-\lambda_0^{s_2})\big] \subseteq D^n(\gamma) V( s_0+n ) 
 =	D^n(\gamma) V\big( n, V(s_0)\cap \Gamma \big),
\end{align*}
where in the last step we used  Lemma~\ref{lem:SupportGrowth}~\eqref{enu:V(n,V(m,M))}.
\end{proof}

\begin{remark}
We note that improving $N_r(T)$ in the previous statement (i.e.\@ making it smaller) for a different testing domain may improve the estimate in Theorem~\ref{thm:upper-bound-estim}.
\end{remark}

\subsection{An upper bound on the rate of convergence} 
\label{subsec:Rate_Converg}

Note that Proposition~\ref{prop:TestingTuple_Existence} provides an explicit growth behavior of the map $r\mapsto N_r(T)$ associated with the testing domain $T=V(s)\cap \Gamma$ for a suitable $s\in\N$. This provides a quantitative version of Lemma~\ref{lem:Patches_event_Legal}.
The concept of linear repetitivity allows us to prove also a quantitative version of Lemma~\ref{lem:LegalPatches_Appear} leading to Theorem~\ref{thm:upper-bound-estim}.
Recall that primitivity of the substitution rule yields that the associated subshift $\Omega(S)$ is linearly repetitive, see Theorem~\ref{thm:Omega(S)}.

\begin{lemma}
\label{lem:LegalPatches_Appear_Quantit}
Let $\Ss=(\Aa,\lambda_0,S_0)$ be a primitive substitution datum over $\Dd=\big( G,d, (D_\lambda)_{\lambda>0}, \Gamma , V  \big)$ with associated substitution map $S$. Let $C_->0$ and $s\in\N_0$ be such that $\lambda_0>1$ is sufficiently large relative to $V$ with respect to $C_-$ and $s$. Denote the linear repetitivity constant of $\Omega(S)$ by $\CLR\geq 1$. If $r\geq 1$, then we have the implication
$$
n\geq L_r:=\frac{\log\left(\frac{\CLR}{C_-}\cdot  r\right)}{\log(\lambda_0)} + s
	\qquad\Longrightarrow\qquad
W(S)_{B(e,r)}\subseteq W(S^n(\omega_0)),\qquad \omega_0\in\Aa^\Gamma.
$$
\end{lemma} 
	
\begin{proof}
Let $n\geq L_r$, $\omega_0\in\Aa^\Gamma$ and set $a:=\omega_0(e)$ and $m:=n-s$. Note that the following considerations do not depend on the specific letter $a$. Since $\lambda_0$ is sufficiently large relative to $V$ with respect to the constant $C_->0$ and $s$, there is a $z\in\Gamma$ with 
$$
B\big(D^m(z), C_- \lambda_0^m\big)
	= D^m\big[B(z,C_-)\big]
	\subseteq V(m+s)
	= V(n).
$$
By Proposition~\ref{prop:SubstitutionMap}~\eqref{enu:support_Sn_M}, $V(n)\cap\Gamma$ is the support of $S^n(a)$. 
Since $n\geq L_r$, we conclude $m\log(\lambda_0)\geq \log\left(\frac{\CLR r}{C_-}\right)$ implying $\lambda_0^m C_-\geq \CLR r$. Thus, $B\big(D^m(z),\CLR r\big)\subseteq V(n)$ follows and so $$Q\prec S^n(a)|_{B(D^m(z),\CLR r)}$$ 
holds for all $Q\in W(S)_{B(e,r)}$ since $\CLR$ is the linear repetititivity constant of $\Omega(S)$.
Hence, $W(S)_{B(e,r)}\subseteq W(S^n(\omega_0))$ follows using $a:=\omega_0(e)$.
\end{proof}

Next we show explicit upper bounds on the convergence rates of the subshifts.

\begin{proposition}
\label{prop:ConvergenceRate}
Consider a dilation datum $\Dd=\big( G,d, (D_\lambda)_{\lambda>0}, \Gamma , V  \big)$ and a primitive substitution datum $\Ss=(\Aa,\lambda_0, S_0)$. Let $S$ be the associated substitution map. Suppose
\begin{itemize}
\item $n_0\in\N_0$ is chosen such that $W(S^{n_0}(\omega_0))_{T} \subseteq W(S)$;
\item there is a testing domain $T$, a constant $C_T>0$ such that for $r\geq 1$, $N_r(T):=\frac{\log(rC_T )}{\log(\lambda_0)}$ satisfies Definition~\ref{def:TestingTuple}~\eqref{enu:TestingDomain};
\item $\CLR\geq 1$ is the linear repetitivity constant of $\Omega(S)$;
\item $C_->0$ and $s\in\N_0$ are chosen such that $\lambda_0$ is sufficiently large relative to $V$ with respect to $C_-$ and $s$.
\end{itemize}
Then  
\[ 
\delta_H\big( \overline{\Orb(S^n(\omega_0))}, \Omega(S) \big)
	< C \frac{1}{\lambda_0^n},
	\qquad n> M_1,
\]
holds where $C:=\max\left\{ \frac{\CLR}{C_-}\lambda_0^s,\,  C_T \lambda_0^{n_0} \right\}$ and $M_1:=\frac{\log(C)}{\log(\lambda_0)}$.
\end{proposition}

\begin{proof}
For $r\geq 1$, define 
$$
M_r:= \max\left\{ \frac{\log\left(\frac{\CLR}{C_-} \cdot r\right)}{\log(\lambda_0)} + s,\, \frac{\log(rC_T  )}{\log(\lambda_0)}+n_0 \right\} 
	= \max\left\{ \frac{\log\left(\frac{\CLR}{C_-} \cdot r\right)}{\log(\lambda_0)} + s,\, N_r(T)+n_0 \right\}. 
$$
Since $W(S^{n_0}(\omega_0))_{T} \subseteq W(S)$, Lemma~\ref{lem:Patches_event_Legal} (applied to $S^{n_0}(\omega_0)$) and Lemma~\ref{lem:LegalPatches_Appear_Quantit} imply 
$$
W(S^n(\omega_0))_{B(e,r)}=W(S)_{B(e,r)},\qquad n\geq M_r.
$$
Thus, equation~\eqref{eq:HausdorffMetric_Dictionaries} leads to
$$
\delta_H\big(\overline{\Orb(S^n(\omega_0))},\Omega(S)\big) \leq \frac{1}{r+1},\qquad n\geq M_r.
$$
Since $M_r$ is continuous and monotonically increasing to infinity, for each $n> M_1$, there is an $r_n> 1$ such that $n=M_{r_n}$. 
We have two cases for $n$ since $M_{r_n}$ is defined by a maximum of two terms. A short computation yields 
$$
r_n=\frac{C_-}{\CLR \lambda_0^s}\lambda_0^n
\qquad\textrm{or}\qquad
r_n=\frac{1}{C_T\lambda_0^{n_0}} \lambda_0^n.
$$

Combining these computations with the previous considerations implies 
$$
\delta_H\big(\overline{\Orb(S^n(\omega_0))},\Omega(S)\big)
	\leq \frac{1}{r_n+1}
	< \frac{1}{r_n}
	\leq \max\left\{ \frac{\CLR}{C_-} \lambda_0^s, C_T \lambda_0^{n_0}\right\} \frac{1}{\lambda_0^n},
\qquad
n\geq M_1.
$$
\end{proof}

\begin{proof}[Proof of Theorem~\ref{thm:upper-bound-estim}]
Recall the assumptions in the statement. Let $S$ be a substitution map of a primitive substitution with dilation datum $\Dd=(G,d,(D_\lambda)_{\lambda>0},\Gamma,V)$ and associated stretch factor $\lambda_0>1$ that is sufficiently large relative to $V$ with respect to a constant $C_->0$ and $s_1\in\N_0$. 
Fix $\delta>0$ and choose $s_2\in\N$ such that $V \cdot B(e,\delta) \subseteq B\big( e ,\lambda_0^{s_2} C_-\big)$, as in Proposition~\ref{prop:TestingTuple_Existence}.
Then $T:=V(s_1+s_2)\cap\Gamma$ is a testing domain such that for $r\geq 1$, $N_r(T):=\frac{\log(rC_T )}{\log(\lambda_0)}$ with $C_T:=\frac{1}{\delta}$ satisfies Definition~\ref{def:TestingTuple}~\eqref{enu:TestingDomain}, see Proposition~\ref{prop:TestingTuple_Existence}.
Primitivity of the substitution implies that $\Omega(S)$ is linearly repetitive with linear repetitivity constant $\CLR\geq 1$, see Theorem~\ref{thm:Omega(S)}. 

Set $n_0:=|\Aa^T|\cdot N_T$. Let $\omega_0\in\Aa^\Gamma$ be such that it satisfies one of the equivalent conditions in Theorem~\ref{thm:charact_convergence}. Thus, Theorem~\ref{thm:charact_convergence} asserts that any path in $G_S(T,N_T)$ starting in a vertex $W(\omega_0)_T\subseteq \Aa^T$ has length strictly less than $|\Aa^T|$. Thus, Lemma~\ref{lem:paths-ileg} yields $W(S^{n_0}(\omega_0))_{T} \subseteq W(S)$.
Now the desired claim follows from Proposition~\ref{prop:ConvergenceRate}. Note that $n_0$ does not depend on $\omega_0$.
\end{proof}

In the case of block substitutions, we can estimate the constants more explicitly using Proposition~\ref{prop:test-euclid} instead of Proposition~\ref{prop:TestingTuple_Existence}. Here we take of advantage that the testing domain is small and can be explicitly computed.

\begin{cor}
\label{cor:RateConv_Zd}
Let $\Aa$ be a finite set and $\mv=(m_1,\ldots,m_d)\in\N^d$ with $m_j>1$ for all $1\leq j\leq d$ and $S_0:\Aa\to\Aa^{\Kmv}$ be a primitive block substitution with $\lambda_0:=\min\{m_j\,|\, 1\leq j\leq d\}\geq 2$. Let $S$ be the associated substitution map of the block substitution and $T=\{0,1\}^d$. 

If $\omega_0\in\Aa^\Gamma$ satisfies $W(\omega_0)_{T}\subseteq W(S)$, then 
$$
\delta_H\big(\overline{\Orb(S^n(\omega_0))},\Omega(S)\big)
	\leq \frac{C}{\lambda_0^n}
	,\qquad n\geq \frac{\log(C)}{\log(\lambda_0)}
$$
where 
$$
C:= \frac{\CLR}{2\lambda_0-3}\lambda_0^4
$$
\end{cor}

\begin{proof}
Using Proposition~\ref{prop:blockSubst_DilationDatum},  $\Dd:=\big(\R^d,d_{\mv},(D_\lambda)_{\lambda>0},\Z^d,\big[-\frac{1}{2},\frac{1}{2}\big)^d\big)$ is a dilation datum and $\lambda_0$ is sufficiently large relative to $V$ with respect to the constant $C_-= 2\lambda_0-3$ and $s=4$.
By assumption we have $W\big(S^{n_0}(\omega_0)\big)_{T}\subseteq W(S)$ for $n_0=0$. Furthermore, Proposition~\ref{prop:test-euclid} asserts that $(T,1)$ is a testing tuple with $T=\{0,1\}^d$ and for each $r\geq 1$, $N_r(T):= \frac{\log(2\lambda_0r)}{\log(\lambda_0)}$ satisfies Definition~\ref{def:TestingTuple}~\eqref{enu:TestingDomain}. Hence, Proposition~\ref{prop:ConvergenceRate} with $C_T=2\lambda_0$ and $n_0=0$ implies the desired claim using $\CLR\geq 1$ and $\lambda_0\geq 2$.
\end{proof}

\subsection{A lower bound on the rate of convergence}
\label{subsec:LowerBounds}

In this subsection, we will prove a lower bound on the rate of convergence of $\delta_H\big(\overline{\Orb(S^n(\omega_0))},\Omega(S)\big)$ if $\omega_0$ is periodic. To this end, we need the concept of lower box counting dimension of a subshift $\Omega$. If $\Omega \in \Jj$ is a subshift, then the {\em lower box counting dimension} is defined by
$$
\lodim(\Omega) 
	:= \liminf_{r\to\infty} \frac{\log\big( \vert W(\Omega)_{B(e,r)}\vert \big)}{\log(r)}
	= \liminf_{r\to\infty} \frac{\log\big( \vert W(\Omega)_{B(e,r)}\vert \big)}{-\log(\frac{1}{2(r+1)})}.
$$

Let us shortly explain why we call this the lower box counting dimension. 
Recall that $\Aa^\Gamma$ is a totally disconnected compact metric space with metric $\dc$ defined in equation~\eqref{eq:config-metric}. In order to define the box counting dimension, one counts the number of sets of diameter at most $\frac{1}{2(r+1)}$ that are needed to cover the set $\Omega$. By the choice of the metric this is exactly the number of different patches with support $B(e,r)$ that $\Omega$ admits, namely the patch counting function $\vert W(\Omega)_{B(e,r)}\vert$. 
One can show that for a linearly repetitive subshift $\Omega$, the patch counting function is at most of the order  $r^\kappa$ where $\kappa$ is the homogeneous dimension defined in Section~\ref{subsec:Substitutions_general}. This shows that the box counting dimension of $\Omega(S)$ is at most $\kappa$. 
We point out at this point that $\lodim(\Omega)$ is also called {\em lower power entropy}, see e.g. \cite[Section 2.5]{Pet16}. 
We note that due to the choice of metric in equation~\eqref{eq:config-metric}, this is not the standard box dimension one considers on subshifts. In fact, the standard Hausdorff dimension of $\Omega(S)$ on $\mathbb{Z}^d$ is $0$, since the entropy of a subshift is the same as its Hausdorff dimension, cf. \cite[Theorem 4.2]{Sim14} or \cite[Proposition III.1]{Fur69}.

\begin{proposition} \label{prop:lower-bound-estim}
Let $S$ be the substitution map of a substitution on $\Aa^\Gamma$ with dilation datum $\Dd=\big( G,d, (D_\lambda)_{\lambda>0}, \Gamma , V  \big)$ and primitive substitution datum $\Ss=(\Aa,\lambda_0, S_0)$. Let $\kappa>0$ be the homogeneous dimension of $\Gamma$.
If $\lodim(\Omega(S))>0$, then for each periodic $\omega_0\in\Aa^\Gamma$, there exists a constant $C(\omega_0)>0$ such that
\[ 
\delta_H \big( \Orb(S^n(\omega_0)), \Omega(S)\big) 
	\geq  C(\omega_0) \left(\frac{1}{\lambda_0^n}\right)^{\frac{\kappa}{\lodim(\Omega(S))}}  \quad \mbox{ for all } n \in \N.
\]
\end{proposition}

\begin{proof}
Since $\omega_0$ is periodic, there is a finite set $M\subseteq\Gamma$ such that $\Orb(\omega_0)=\{\eta\omega_0 \,|\, \eta\in M\}$. By Proposition~\ref{prop:Periodic_Sn}, we have
$$
\Orb(S^n(\omega_0))=\big\{ \eta S^n(\omega_0) \,|\, \eta \in (D^n[V]\cap\Gamma)\cdot D^n[M] \big\}, \qquad n\in\N.
$$
Since $V$ is relatively compact, there is an $r_+>0$ such that $V\subseteq B(e,r_+)$. Thus, $D^n[V]\cap\Gamma\subseteq B(e,r_+ \lambda_0^n)$.
Since $\Gamma$ has exact polynomial growth, see Section~\ref{subsec:Substitutions_general}, there is a constant $C_\Gamma>0$ such that $|B(e,r) \cap \Gamma|\leq C_\Gamma r^\kappa$ for all $r\geq 1$.
With all this at hand, we estimate
$$
\vert \Orb(S^n(\omega_0))\vert 
	\leq \vert D^n[M]\vert \cdot \vert D^n[V]\cap\Gamma \vert
	\leq \vert M \vert \cdot \vert B(e,r_+\lambda_0^n) \cap \Gamma \vert
	\leq \vert M \vert \cdot C_\Gamma \cdot r_+^\kappa \cdot \lambda_0^{n\kappa}.
$$
Suppose that $\delta_H\big( \Orb(S^n(\omega_0)), \Omega(S)\big)  < \frac{1}{r+1}$ for $r> 1$. Then $W(S^n(\omega_0))_{B(e,r)}=W(S)_{B(e,r)}$ follows from equality~\eqref{eq:HausdorffMetric_Dictionaries}. Note that if $P,Q\in W(S^n(\omega_0))_{B(e,r)}$ are different, then there are two different $\omega_P,\omega_Q\in \Orb(S^n(\omega_0))$ satisfying $\omega_P|_{B(e,r)}=P$ and  $\omega_Q|_{B(e,r)}=Q$. Hence, $\vert \Orb(S^n(\omega_0))\vert \geq \vert W(S^n(\omega_0))_{B(e,r)} \vert$ follows. Combining this with the previous considerations yields
$$
\vert M \vert \cdot C_\Gamma \cdot r_+^\kappa \cdot \lambda_0^{n\kappa} 
	\geq \vert \Orb(S^n(\omega_0))\vert
	\geq \vert W(S^n(\omega_0))_{B(e,r)} \vert
	= \vert W(S)_{B(e,r)} \vert.
$$
The last term can be estimated using the lower box counting dimension. In particular, there is a constant $C_B>0$ such that $\vert W(S)_{B(e,r)} \vert\geq C_B r^{\lodim(\Omega(S))}$. Hence, we obtain
$$
\vert M \vert \cdot C_\Gamma \cdot r_+^\kappa \cdot \lambda_0^{n\kappa} \geq C_B r^{\lodim(\Omega(S))}.
$$
To conclude, we proved the implication
$$
\delta_H\big( \Orb(S^n(\omega_0)), \Omega(S)\big)  < \frac{1}{r+1} \; \textrm{for} \; r>1
	\qquad\Longrightarrow\qquad
r\leq \left( \frac{|M|\lambda_0^{n\kappa}}{C_1} \right)^{\frac{1}{\lodim(\Omega(S))}}=:R(n)
$$ 
where $C_1:=\frac{C_B}{C_\Gamma \cdot r_+^\kappa}$. Since $\lambda_0>1$, $R(n)$ diverges implying $R_0:=\min_{n\in\N} R(n)$ exists and is positive. Thus, there is a $C_2>1$ (independent of $n\in\N$) such that $(C_2-1)R_0>1$. Hence, $(C_2-1)R(n)>1$ follows for all $n\in\N$ implying
$$
r(n):=C_2R(n)-1>R(n).
$$ 
Then the previous considerations imply by contraposition that
$$
\delta_H\big( \Orb(S^n(\omega_0)), \Omega(S)\big)  
	\geq \frac{1}{r(n)+1}
	\geq \frac{1}{C_2}
		\left(\frac{C_1}{|M|}\right)^{\frac{1}{\lodim(\Omega(S))}}
		\left(\frac{1}{\lambda_0^n}\right)^{\frac{\kappa}{\lodim(\Omega(S))}} 
$$
proving the claim with 
\[
C(\omega_0) := \frac{1}{C_2}
		\left(\frac{C_1}{|M|}\right)^{\frac{1}{\lodim(\Omega(S))}}.
\]
\end{proof}

\begin{remark}
For $\Gamma = \mathbb{Z}^d$ one can adjust \cite[Lemma~2.2]{Lenz04} to the symbolic setting in order to conclude $\vert W(\Omega)_{B(e,r)}\vert \geq Cr^d$ if $r\geq 1$ is large enough. Thus the lower box counting dimension of $\Omega(S)$ satisfies  $\lodim(\Omega(S))=d = \kappa$ in this case.  
We conclude that if $S$ is a substitution map of a primitive non-periodic substitution on $\Aa^{\Z^d}$ (or a block substitution), then Proposition~\ref{prop:lower-bound-estim} asserts
\[ 
\delta_H \big( \Orb(S^n(\omega_0)), \Omega(S)\big) 
	\geq  C(\omega_0) \frac{1}{\lambda_0^n}
\]
for all periodic $\omega_0\in\Aa^{\Z^d}$, where $C(\omega_0)>0$ is a constant depending on $\omega_0$. This shows in particular that our upper bound proven in Proposition~\ref{prop:ConvergenceRate} is asymptotically optimal in this case. 
\end{remark}

\section{Algorithm to reduce testing domains} 
\label{sec:Testing_Tuple_reduction}

In order to apply Theorem \ref{thm:charact_convergence} to specific examples, it is important to compute the graphs $G_S(T,N_T)$. The smaller the testing domain is, the smaller the graphs become. Therefore minimizing the testing domain is computationally crucial. We provide a general algorithm for this purpose and apply it to the Heisenberg group. The Python source code, along with its documentation can be found in the GitHub repository \cite{TestHeisen24}.

\medskip

We remind the reader that a testing tuple always exists by Proposition~\ref{prop:TestingTuple_Existence}. Furthermore, a testing tuple depends on a dilation datum $\Dd$ with associated stretch factor $\lambda_0>1$. We continue using the notation $D:=D_{\lambda_0}$. 

\begin{lemma} \label{lem:test-lemma}
Consider a dilation datum $\Dd$ with associated stretch factor $\lambda_0>1$. Let $T_0\subseteq \Gamma$ be a testing domain of the substitution, and let $T\subseteq \Gamma$ be finite. Then $T$ is a testing domain of the substitution if and only if there exists an $N_0\in \mathbb{N}$ such that for all $x\in \Gamma$, there exists a $\gamma=\gamma(x)\in \Gamma$ satisfying
\begin{equation} \label{eq:test-lemma-contain}
x T_0\subseteq D^{N_0}(\gamma)  V(N_0,T).
\end{equation}
\end{lemma}

\begin{proof}
If $T$ is a testing domain then \eqref{eq:test-lemma-contain} follows immediately from  Definition~\ref{def:TestingTuple}~\eqref{enu:TestingDomain}, using that $T_0$ is finite. Thus, it is contained in a ball $B(e,r_0)$ for some $r_0>0$. Now suppose the inclusion~\eqref{eq:test-lemma-contain} holds with $N_0\in\N$ and we prove that $T$ is a testing domain. Let $r>0$. Since $T_0$ is a testing domain, there is an $N_r(T_0)\in\N$ satisfying Definition~\ref{def:TestingTuple}~\eqref{enu:TestingDomain}. Set $N_r(T):=N_r(T_0)+N_0$. Let $x\in\Gamma$ and $n\geq N_r(T)$. Since $n-N_0\geq N_r(T_0)$, and $T_0$ is a testing domain, there is an $\eta=\eta(x,n-N_0)\in\Gamma$ such that 
$$
xB(e,r)\subseteq D^{n-N_0}(\eta) V(n-N_0,T_0) = V(n-N_0,\eta T_0),
$$
where we used Proposition~\ref{prop:SubstitutionMap}~\eqref{enu:support_Sn_Equivariant} in the last equality. By equation~\eqref{eq:test-lemma-contain}, there is an $\gamma=\gamma(\eta)\in\Gamma$ such that $\eta T_0 \subseteq D^{N_0}(\gamma)V(N_0,T)=V(N_0,\gamma T)$. Combined with the previous considerations, Proposition~\ref{prop:SubstitutionMap}~\eqref{enu:support_Sn_Equivariant} and Lemma~\ref{lem:SupportGrowth}~\eqref{enu:V(n,V(m,M))},\eqref{enu:M_subset_M'} imply that
$$
xB(e,r)\subseteq V(n-N_0,\eta T_0) \subseteq V\big( n-N_0, V(N_0, \gamma T)\cap \Gamma \big)= V(n-N_0+N_0, \gamma T)= D^n(\gamma) V(n,T).
$$
Thus, $T$ satisfies Definition~\ref{def:TestingTuple}~\eqref{enu:TestingDomain}, namely $T$ is a testing domain.
\end{proof}

We aim at showing that it is sufficient to consider only finitely many $x$ to invoke Lemma~\ref{lem:test-lemma}, see Lemma~\ref{lem:test-lemma-quotient} and Proposition~\ref{prop:test-dom-criterion} below. To this end, we use the following observation.

\begin{lemma}
\label{lem:V(n)_fundamental_domain}
Consider a dilation datum $\Dd$ with associated stretch factor $\lambda_0>1$. Then for $n\in\N$, the sets $D^n[V]$ and $V(n)$ (defined in equation~\eqref{eq:recursive_V(n)}) are fundamental domains of $D^n[\Gamma]$.
\end{lemma}

\begin{proof}
We first prove that $D^n[V]$ is a fundamental domain for $D^n[\Gamma]$. Since $D$ is an automorphism, observe
$$
D^n[\Gamma]\cdot D^n[V] = D^n[\Gamma\cdot  V] = D^n[G] = G.
$$
In addition, for $\gamma,\eta\in\Gamma$, 
$$
D^n(\gamma)D^n[V]\cap D^n(\eta)D^n[V]\neq\emptyset
	\quad\Longleftrightarrow\quad D^n[\gamma V]\cap D^n[\eta V]\neq \emptyset
		\quad\Longleftrightarrow\quad \gamma V\cap \eta V\neq \emptyset
$$ 
holds proving the claim. Next, we show that $V(n)$ is a fundamental domain of $D^n[\Gamma]$ by induction. For the induction base, observe that 
$$
D^1[V] = D[V] = D\big[(V\cap\Gamma)\cdot V\big]=V(1).
$$
Therefore, $V(1)$ is a fundamental domain of $D(\Gamma)$ by the previous considerations. For the induction step, suppose that $V(n)$ is a fundamental domain of $D^n[\Gamma]$. Thus, we have
\[
 D^n[\Gamma]\cdot(V(n)\cap\Gamma)= \bigcup_{\gamma \in \Gamma} D^n(\gamma)\big( V(n) \cap \Gamma \big) = \Big( \bigcup_{\gamma \in \Gamma} D^n(\gamma)V(n) \Big) \cap \Gamma  = D^{n}[\Gamma]\cdot V(n) \cap \Gamma = \Gamma.
\]
Since $D$ is an automorphism, we conclude
\begin{align*}
D^{n+1}[\Gamma]\cdot V(n+1) 
	&= D^{n+1}[\Gamma]\cdot  D\big[ (V(n)\cap\Gamma)\cdot V\big]
	= D\big[{D^n[\Gamma]\cdot \big( (V(n)\cap\Gamma)}\cdot V \big)\big] \\
	& = D\big[\underbrace{\big(D^n[\Gamma]\cdot (V(n)\cap\Gamma) \big)}_{= \Gamma}\cdot V\big] 
		= D\big[ G\big]	= G.
\end{align*}
In addition, Proposition~\ref{prop:SubstitutionMap}~\eqref{enu:support_Sn_Equivariant} and Lemma~\ref{lem:SupportGrowth}~\eqref{enu:M_cap_M'} imply for distinct $\gamma,\eta\in\Gamma$  that 
$$
D^{n+1}(\gamma)V(n+1)\cap D^{n+1}(\eta)V(n+1)= V(n+1,\{\gamma\})\cap V(n+1,\{\eta\}) = \emptyset.
$$
This finishes the induction.
\end{proof}

\begin{lemma} 
\label{lem:test-lemma-quotient}
Consider a dilation datum $\Dd$ with associated stretch factor $\lambda_0>1$. Let $n\in \N$ and let $V_n$ be a relatively compact fundamental domain for $D^n[\Gamma]$. Then, the following statements are equivalent for finite sets $T_1,T_2\subseteq \Gamma$.
\begin{enumerate} [(i)]
	\item For every $x\in \Gamma$, there exists a $\gamma(x)\in \Gamma$ satisfying $xT_1\subseteq D^n(\gamma(x))V(n,T_2)$.
	\item For every $x\in V_n\cap \Gamma$, there exists a $\gamma(x)\in \Gamma$ satisfying $xT_1\subseteq D^n(\gamma(x))V(n,T_2)$.
\end{enumerate}
\end{lemma}

\begin{proof}
Clearly $(i)$ implies $(ii)$. Suppose $(ii)$ holds and let $x\in \Gamma$. Since $G=\bigsqcup_{\eta\in \Gamma} D^n(\eta)V_n$, there exist $\eta\in \Gamma$ and $y\in V_n\cap \Gamma$, such that $x=D^n(\eta)y$. Applying $(ii)$ for $y$, there exists a $\gamma(y)\in \Gamma$ satisfying $yT_1\subseteq D^n(\gamma(y))V(n,T_2)$. Hence,
\[ 
xT_1= D^n(\eta)yT_1\subseteq D^n\big(\eta \gamma(y)\big)V(n,T_2). 
\]
Setting $\gamma(x):= \eta \gamma(y)$ shows that $(i)$ holds.
\end{proof}

According to \cite[Theorem~5.13]{BHP21-Symbolic}, for every dilation datum $\Dd$ with associated stretch factor $\lambda_0>1$, there is a constant $C_+>0$ satisfying $V(n)\subseteq B( e, \lambda_0^{n} C_+ )$ for all $n\in\N$. This will be used in the following proposition to obtain a computable sufficient condition for testing domains.

\begin{proposition} \label{prop:test-dom-criterion}
Consider a dilation datum $\Dd$ with associated stretch factor $\lambda_0>1$ and $C_+>0$ as above.
Let $T_0\subseteq \Gamma$ be a testing domain of the substitution with $e\in T_0$. Then the following assertions are equivalent for a finite set $T\subseteq \Gamma$ and $R_T>0$ satisfying $T\subseteq B(e,R_T)$.
\begin{enumerate}[(i)]
\item The set $T$ is a testing domain of the substitution.
\item \label{enu:test-domain-algo-cond} There exists an $N_0\in \N$ such that for all $x\in  D^{N_0}[V]\cap \Gamma$, there is a $\gamma_x\in B\big(D^{-N_0}(x), R_T+C_+ \big)  \cap \Gamma$ such that $x T_0\subseteq D^{N_0}(\gamma_x)  V(N_0,T)$.
\end{enumerate}
\end{proposition}

\begin{proof}
For each $n \in \N$, $D^n[V]$ is a fundamental domain for $D^n[\Gamma]$ by Lemma~\ref{lem:V(n)_fundamental_domain}. By Lemma~\ref{lem:test-lemma} and Lemma~\ref{lem:test-lemma-quotient}, the statement $(i)$ is equivalent to the fact that for all $x\in  D^{N_0}[V]\cap \Gamma$, there exists $\gamma_x\in \Gamma$ such that $x T_0\subseteq D^{N_0}(\gamma_x)  V(N_0,T)$.
Thus, it suffices to prove that $\gamma_x$ can be chosen in the set $B\big(D^{-N_0}(x), R_T+C_+ \big)$.

Let us first note that $x\in D^{N_0}(\gamma_x)  V(N_0,T)$,  since $e\in T_0$.  
Proposition~\ref{prop:substitutionmap}~\eqref{enu:support_Sn_Equivariant} and Lemma~\ref{lem:SupportGrowth}~\eqref{enu:M_cap_M'} imply $V(N_0,T)\subseteq \bigsqcup_{\eta \in T} D^{N_0}(\eta) V(N_0)$.
Recall that $T\subseteq B(e,R_T)$, $D^{N_0}\big[B(e,r)\big]=B(e,\lambda_0^{N_0}r)$ for $r>0$ and $V(N_0)\subseteq B\big( e, \lambda_0^{N_0}C_+ \big)$. Relying on these facts, we conclude that 
\begin{equation*}
	\begin{aligned}
	  x \in D^{N_0}(\gamma_x)  V(N_0,T) 
	  	\subseteq D^{N_0}(\gamma_x)  \Bigg( \bigsqcup_{\eta \in T} D^{N_0}(\eta) V(N_0) \Bigg)  
		& \subseteq  D^{N_0}(\gamma_x) \Bigg( B\big(e,\lambda_0^{N_0}R_T\big)\cdot B\big(e,\lambda_0^{N_0}C_+\big) \Bigg) \\
		& \subseteq D^{N_0}(\gamma_x) D^{N_0}\bigg[ B\big(e, R_T+C_+\big) \bigg],
	\end{aligned}
\end{equation*}
using that the metric $d$ is left-invariant for the last inclusion. 
Hence, $\gamma_x \in  D^{-N_0}(x)\big(B(e, R_T+C_+ )  \big)^{-1}$ follows.
Since $d$ is a left-invariant metric, we have $B(e,r)^{-1}=B(e,r)$ implying
\[ 
\gamma_x\in D^{-N_0}(x)\cdot B(e, R_T+C_+ ) = B\big(D^{-N_0}(x), R_T+C_+ \big).
\]
\end{proof}

We now use Proposition \ref{prop:test-dom-criterion} to propose an algorithm (Algorithm~\ref{Algo-test-reduc}) that checks whether a suspected finite set is a testing domain.  
Let $\Dd$ be a dilation datum with associated stretch factor $\lambda_0>1$ being sufficiently large relative to $V$ with respect to a constant $C_->0$ and $s_1 \in \N_0$.
Choose $s_2\in\N$ such that $\overline{V}\subseteq B(e, C_- \lambda_0^{s_2})$. Then, there exists some $\delta>0$ such that $V\cdot B(e,\delta)\subseteq B(e,\lambda_0^{s_2} C_-)$.
Thus, $T_0:=V(s_1+ s_2)\cap \Gamma$ is a testing domain by  Proposition~\ref{prop:TestingTuple_Existence}.

Let $T\subseteq \Gamma$ be finite and $R_T>0$ such that $T\subseteq B(e,R_T)$. Then we can use condition \eqref{enu:test-domain-algo-cond} of Proposition~\ref{prop:test-dom-criterion} for a fixed $N_0\in\N$ to check if $T$ is also a testing domain. For given $N_0\in\N$, the following algorithm returns the value ``true'' if $T$ is a testing domain. 
 
Note that the Algorithm~\ref{Algo-test-reduc} provides only a sufficient criteria to check if a given $T$ is a testing domain. 
In the case of the Heisenberg group the algorithm can be further refined to make it more efficient, see \cite{TestHeisen24}. Note that for computational purposes it might be more efficient to run the algorithm recursively. More precisely, one might reduce the size of the initial testing domain $T_0$ in an iterative way. We do so in the case of the Heisenberg group 
and obtain the following result. Recall the notations introduced in Section~\ref{subsec:Heisenberg_Example}.

\begin{proposition} \label{prop:Heisen-test-dom}
The tuple $(T',1)$ with $T'=\{ -2,0 \}^2\times\{ -6,-4,...,4,6 \}$ is a testing tuple for dilation datum $\Dd_{H}=\big( \HeiR, d, (D_\lambda)_{\lambda>0}, \HeiZ, [-1,1)^3 \big)$ with associated stretch factor $\lambda_0=4$.
\end{proposition}

\begin{proof}
Recall that $T=V(1)\cap \Gamma$ is a testing domain by Proposition \ref{prop:Periodic_Heisenberg}.
Define $T_1:= \{ -2, 0 \}^2 \times \{ -12, 10 , ... , 10 ,12 \}$.

Then the applications \emph{Alg}($\Dd_H$; 4, 1, $T$, $T_1$) and \emph{Alg}($\Dd_{H}$; 4, 1 $T_1$, $T'$) of Algorithm \ref{Algo-test-reduc} return the value ``true''. 
Thus, $T'$ is a testing domain. Since $T' \subseteq V(1)\subseteq V(1,T')$ holds, we conclude that $N_{T'}=1$ satisfies Definition~\ref{def:TestingTuple}~\eqref{enu:N_T} for the set $T'$. Hence, $(T',1)$ is a testing tuple.
\end{proof}

Proposition~\ref{prop:Heisen-test-dom} enhances possible computations which one desires to perform in the case of Heisenberg group substitution. It reduces the size of the testing domain from $|V(1)\cap\Gamma|=4^4=256$ to $|T'|=2\cdot2\cdot 7= 28$.

\begin{algorithm}[H]
		\caption{Algorithm to verify testing domain, denoted \emph{Alg}($\Dd$; $\lambda_0$, $N_0$, $T_0$, $T$).} \label{Algo-test-reduc}
		\begin{algorithmic}[1]
			\State \emph{Input:} 
			\item \hspace{0.5cm} $\bullet$ a dilation datum $\Dd$ 
			\item \hspace{0.5cm} $\bullet$ a stretch factor $\lambda_0$ associated with $\Dd$
			\item \hspace{0.5cm} $\bullet$ a testing domain $T_0$ 
			\item \hspace{0.5cm} $\bullet$ a finite set $T\subseteq\Gamma$
			\item \hspace{0.5cm} $\bullet$ an iteration number $N_0\in\N$
			\State \emph{Output:}  \textit{True} if condition~\eqref{enu:test-domain-algo-cond} in Proposition~\ref{prop:test-dom-criterion} holds with the given iteration number $N_0$. \textit{False} otherwise.
			\State Compute the set $ V(N_0,T)\cap \Gamma$ via the recursion given in equation~\eqref{eq:recursive_V(n)}.
			\State Compute the set $ D^{N_0}[V]\cap \Gamma$.
			\State Define the radius $R_T:=\max_{\gamma\in T} d(e,\gamma)+1$.
			\State Compute the set $ B(e, R_T+C_+ ) \cap \Gamma$.
			\ForAll {$x \in D^{N_0}[V]\cap \Gamma$:} 
			\State bool=False
			\State Compute $A:=x T_0$.
			\State Compute $ D^{-N_0}(x)\big( B(e, R_T+C_+ ) \big)  \cap \Gamma$.
			\ForAll{$\gamma_x \in D^{-N_0}(x)\big( B(e, R_T+C_+ )\big) \cap \Gamma$}
			\State Compute $B:=D^{N_0}(\gamma_x)V(N,T)$.
			\If{$A\subseteq B$}
			\State bool=True
			\State \emph{break} inner loop		
			\EndIf
			\EndFor
			\If{bool=False}
				\State \Return \textit{False}
			\EndIf
			\EndFor
			\State \Return \textit{True}
		\end{algorithmic}
	\end{algorithm}

	\bibliography{Biblio-Sub}
	\bibliographystyle{alpha}
	
\end{document}